\documentclass{siamltex}
\usepackage{lineno,hyperref}
\usepackage{multirow}
\usepackage[english]{babel}
\usepackage{amsmath}
\usepackage{bm}
\usepackage{mathrsfs}
\usepackage{calligra}
\usepackage{calrsfs}
\usepackage{amsfonts}
\usepackage{amssymb}
\usepackage{graphicx}
\usepackage{tikz}
\usepackage{here}
\usepackage{enumitem}
\usepackage{algorithm}
\usepackage{algorithmic}
\usepackage{graphicx}
\usepackage{amsmath}
\usepackage{here}
\usetikzlibrary{matrix,calc}
\usepackage{caption}
\newcommand{\R}{\mathbb{R}}
\newcommand{\C}{\mathbb{C}}
\newcommand{\K}{\mathbb{K}}
\DeclareMathOperator*{\tNull}{\text{Null}_t}
\DeclareMathOperator*{\Null}{\text{Null}}

\DeclareMathOperator*{\tdiag}{\text{t}-\diag}
\DeclareMathOperator*{\tspan}{\text{span}_t}
\DeclareMathOperator*{\tdet}{Det}
\DeclareMathOperator*{\tdim}{\dim_t}

\newtheorem{theo}{Theorem}

\newtheorem{defi}[theo]{Definition}

\title{Spectral computation with third-order tensors \\ using the \lowercase{t}-product}

\author{A. El Hachimi\footnotemark[2] \thanks{Mohammed VI Polytechnic University, Green 
City, Morocco.}
\and 
K. Jbilou\footnotemark[1] \thanks{Universit\'e du Littoral Cote d'Opale, LMPA, 50 rue 
F. Buisson, 62228 Calais-Cedex, France.}
\and 
A. Ratnani\footnotemark[1]
\and 
L. Reichel\thanks{Department of Mathematical Sciences, Kent State University, Kent, OH 
44242, USA.}
}

\begin{document}
\maketitle 

\begin{abstract}
The tensor t-product, introduced by Kilmer and Martin \cite{kilmer1}, is a powerful tool
for the analysis of and computation with third-order tensors. This paper introduces 
eigentubes and eigenslices of third-order tensors under the t-product. The eigentubes and 
eigenslices are analogues of eigenvalues and eigenvectors for matrices. Properties of
eigentubes and eigenslices are investigated and numerical methods for their computation 
are described. The methods include the tensor power method, tensor subspace iteration, and
the tensor QR algorithm. Computed examples illustrate the performance of these methods.
\end{abstract}

\begin{keywords}
tensor, t-product, eigentube, eigenslice, tensor power method, tensor subspace iteration, 
tensor QR algorithm. 
\end{keywords}

\section{Introduction}
Tensors are high-dimensional generalizations of matrices. They find many applications 
in science and engineering including in image processing, data mining, computer vision, 
network analysis, and the solution of partial differential equations; see
\cite{kilmer-compression,kolda2009tensor,BNKR,boubekraoui,elha1,elha2,BKS,CRZT,elhanr1,
pagerank,RU1} and references therein. Several different tensor products have been proposed
in the literature. The $n$-mode product is one of the most commonly used products; it 
defines the product of a tensor and a matrix of appropriate sizes; see 
\cite{kolda2009tensor,canonical_unfold}. The CP and Tucker tensor decompositions use this 
product; see \cite{kolda2009tensor,BKS2}. The oldest tensor product is probably the 
Einstein product \cite{Einstein_product}; it has recently been applied to color image and 
video processing \cite{elguide2}. The newest tensor product is the t-product for 
third-order tensors, which was proposed by Kilmer et al. \cite{kilmer2013third,kilmer1}. 
This product has found many applications and has been generalized in several ways; see 
\cite{kernfeld2015tensor,kilmer-compression,elha1,elha2,dufrenois,elguide1,elha3,elichi1,
MSL,RU1,RU2}. The t-product allows natural generalizations of many matrix factorizations 
including the singular value decomposition, and the QR and Choleski factorizations to 
third-order tensors; see \cite{kilmer2013third,kilmer1,RU3}.

Several notions of eigenvalues for a tensor have been described in the literature. They 
include the C-eigenvalues, H-eigenvalues, and Z-eigenvalues; see \cite{C-eigenvalue,QL}.
The definitions depend on the tensor product used. The eigenvalues in these definitions 
are scalars. We will use the t-product and introduce eigentubes and eigenslices, which are
analogues for third-order tensors of eigenvalues and eigenvectors for matrices, 
respectively. It is the purpose of this paper to discuss some properties of eigentubes and
eigenslices, and describe algorithms for their computation.

The organization of this paper is as follows: Sections \ref{sec 2} reviews properties of 
the t-product, Section \ref{sec 3} describes some tensor decompositions, and Section 
\ref{sec 4} discusses some properties of eigentubes and eigenslices. Numerical methods for
computing a few or all eigentubes and associated eigenslices of a tensor are presented in 
Section \ref{sec 5}. A few computed examples are presented in Section \ref{sec 6}, and 
Section \ref{sec7} contains concluding remarks.

\section{The t-product for third order tensors}\label{sec 2}
This section reviews definitions and results by Kilmer et al. 
\cite{kilmer2013third,kilmer1} and uses notation from there, as well as from Kolda and 
Bader \cite{kolda2009tensor}. Consider the third-order tensor
$\mathcal{A}=[\mathcal{A}_{i,j,k}]\in\C^{\ell\times p\times n}$. A \emph{slice} of 
$\mathcal{A}$ is a 2D section obtained by fixing any one of the indices. Using MATLAB-type
notation, $\mathcal{A}_{i,:,:}$, $\mathcal{A}_{:,j,:}$, and $\mathcal{A}_{:,:,k}$ stand 
for the $i$th horizontal, $j$th lateral, and $k$th frontal slices of $\mathcal{A}$, 
respectively. We also denote the $j$th lateral slice by $\vec{\mathcal{A}}_j$. It is a 
tensor and is referred to as a \emph{tensor column}. Sometimes it is convenient to 
identify $\vec{\mathcal{A}}_j$ with a matrix. The $k$th frontal slice, 
$\mathcal{A}_{:,:,k}$ is also denoted by $\mathcal{A}^{(k)}$ and is a matrix. A 
\emph{fiber} of a third order tensor $\mathcal{A}$ is a 1D section obtained by fixing any
two of the indices. Thus, $\mathcal{A}_{:,j,k}$, $\mathcal{A}_{i,:,k}$, and 
$\mathcal{A}_{i,j,:}$ denote mode-1, mode-2, and mode-3 fibers (tubes), respectively. We
will use the notation $\K^{\ell\times p}_n=\C^{\ell\times p\times n}$ for the space of 
third-order tensors, $\K^\ell_n=\C^{\ell\times 1\times n}$ stands for the space of lateral
slices, and $\K_n=\C^{1\times 1\times n}$ denotes the space of tubes. Many of the tensors
considered will be \emph{square}, i.e., $\ell=p$.

The inner product of two third-order tensors 
$\mathcal{A},\mathcal{B}\in\K^{\ell\times p}_n$ is given by
\[
\langle\mathcal{A},\mathcal{B}\rangle=\sum_{i,j,k=1}^{\ell,p,n} 
\mathcal{A}_{i,j,k}\bar{\mathcal{B}}_{i,j,k},
\]
where the bar denotes complex conjugation, and the Frobenius norm of $\mathcal{A}$ is 
defined as
\[
\left\Vert\mathcal{A}\right\Vert_F=\langle\mathcal{A},\mathcal{A}\rangle^{1/2}.
\]
We note that a third-order tensor $\mathcal{A}\in\K^{\ell\times p}_n$ can be represented 
as
\[
\mathcal{A}=
\left[\vec{\mathcal{A}}_1,\vec{\mathcal{A}}_2,\ldots,\vec{\mathcal{A}}_p\right].
\]

The Discrete Fourier Transform (DFT) is helpful for the efficient evaluation of the 
t-product of two tensors. The DFT of a vector $v\in\C^n$ is given by
\begin{equation}\label{eq 1}
\widehat{v}=F_n v\in\C^n,
\end{equation} 
where $F_n$ denotes the DFT matrix defined as 
\[
F_n=\left[\left(e^{\frac{2i\pi}{n}}\right)^{(k-1)(j-1)}\right]_{k=1:n, j=1:n}\in
\C^{n\times n}.
\]
While the matrix $F_n$ is not unitary, the scaled matrix $F_n/\sqrt{n}$ is. The 
computation of the vector $\widehat{v}$ by straightforward evaluation of the matrix-vector
product in the right-hand side of \eqref{eq 1} requires $O(n^2)$ arithmetic floating point
operations (flops). It is well known that this flop count is reduced to $O(n\log(n))$ when
using the fast Fourier transform (FFT). This can be done in MATLAB with the command 
$\widehat{v}={\tt fft}(v)$.

Let $\widehat{\mathcal{A}}\in\K^{\ell\times p}_n$ denote the tensor obtained by applying
the DFT along the third dimension of the tensor $\mathcal{A}\in\R^{\ell\times p\times n}$.
Thus, we evaluate the DFT of each tube of $\mathcal{A}$; this can be done in MATLAB with 
the command 
\[
\widehat{\mathcal{A}}={\tt fft}\left(\mathcal{A},[\,],3\right).
\]
To transform $\widehat{\mathcal{A}}$ back to $\mathcal{A}$, we use the MATLAB command
\[
\mathcal{A}={\tt ifft}\left(\widehat{\mathcal{A}},[\,],3\right).
\]
For $\mathcal{A},\mathcal{B}\in\K^{\ell\times p}_n$, we have
\[
\left\Vert\mathcal{A}\right\Vert_F =
\dfrac{1}{\sqrt{n}}\Vert\widehat{\mathcal{A}}\Vert_F,\quad
\langle\mathcal{A},\mathcal{B}\rangle=
\dfrac{1}{n}\langle\widehat{\mathcal{A}},\widehat{\mathcal{B}}\rangle.
\]

The following definitions are required below. For $\mathcal{A}\in\K^{\ell\times p}_n$, we
define the associated block diagonal matrix 
\[
{\tt bdiag}\left(\mathcal{A}\right)=\begin{pmatrix}
\mathcal{A}^{(1)}&  &  &  & \\
& 	\mathcal{A}^{(2)} &  & & \\
&  & \mathcal{A}^{(3)} &  & \\
&  &  &  \ddots  & \\
&  & & &  \mathcal{A}^{(n)}
\end{pmatrix}\in\C^{\ell n\times pn}.
\]
Also introduce the block circulant matrix 
\[
{\tt bcirc}\left(\mathcal{A}\right)=\begin{pmatrix}
\mathcal{A}^{(1)} & \mathcal{A}^{(n)}  & \mathcal{A}^{(n-1)} & \dots & \mathcal{A}^{(2)}\\ 
\mathcal{A}^{(2)} & \mathcal{A}^{(1)}  & \mathcal{A}^{(n)} & \dots & \mathcal{A}^{(3)}\\ 
\vdots & \ddots & \ddots & \ddots & \vdots\\
\mathcal{A}^{(n-1)} & \mathcal{A}^{(n-2)} & \dots & \mathcal{A}^{(1)} & \mathcal{A}^{(n)}\\
\mathcal{A}^{(n)} & \mathcal{A}^{(n-1)} & \dots & \mathcal{A}^{(2)} & \mathcal{A}^{(1)}
\end{pmatrix}\in\C^{\ell n\times pn}.
\]
The operators ${\tt unfold}$ and ${\tt fold}$ are defined as
\[
{\tt unfold}\left(\mathcal{A}\right)=\begin{pmatrix}
\mathcal{A}^{(1)} \\
\mathcal{A}^{(2)}\\
\vdots\\
\mathcal{A}^{(n)}
\end{pmatrix}\in\C^{\ell n\times p} \text{~~and~~} 
{\tt fold}({\tt unfold}\left(\mathcal{A}\right))=\mathcal{A}.
\]
Thus, ${\tt unfold}$ transforms the tensor $\mathcal{A}\in\K^{p\times\ell}_n$ into a 
matrix and ${\tt fold}$ is the inverse operator.

A block circulant matrix can be block diagonalized by using the DFT. Kilmer et al. 
\cite{kilmer2013third} use this property to establish that for
$\mathcal{A}\in\K^{\ell\times p}_n$, we have
\[
\left(F_n \otimes I_\ell\right){\tt bcirc}\left(\mathcal{A}\right)
\left(F_n^H \otimes I_p\right)={\tt bdiag}(\widehat{\mathcal{A}}),
\]
where $\otimes$ stands for the Kronecker product. The following definition is due to 
Kilmer et al. \cite{kilmer2013third}.

\begin{defi}
Let $\mathcal{A}\in\K^{\ell\times q}_n$ and $\mathcal{B}\in\K^{q\times p}_n$. Then the 
t-product of $\mathcal{A}$ and $\mathcal{B}$ is defined as
\begin{equation}\label{tprod}
\mathcal{A}*\mathcal{B}=
{\tt fold}\left({\tt bcirc}\left(\mathcal{A}\right)
{\tt unfold}\left(\mathcal{B}\right)\right)\in\K^{\ell\times p}_n.
\end{equation}
\end{defi}

The evaluation of the t-product by using the above definition is quite costly unless the
tensors $\mathcal{A}$ and $\mathcal{B}$ are very sparse. Kilmer et al. 
\cite{kilmer2013third} therefore propose to evaluate the t-product \eqref{tprod} with the 
aid of the FFT. The following result makes this possible.

\begin{proposition}(\cite{kilmer2013third})
Let $\mathcal{A}\in\K^{\ell\times q}_n$ and $\mathcal{B}\in\K^{q\times p}_n$. Then the 
tensor $\mathcal{C}=\mathcal{A}*\mathcal{B}$ can be determined from
\[
{\tt bdiag}(\widehat{\mathcal{C}})={\tt bdiag}(\widehat{\mathcal{A}})\,
{\tt bdiag}(\widehat{\mathcal{B}}).
\]
\end{proposition}

We briefly consider the special case of evaluating the t-product of real tensors 
$\mathcal{A}\in\R^{\ell\times q\times n}$ and $\mathcal{B}\in\R^{q\times p\times n}$. It 
is well known that for $v\in\R^n$, the vector $\widehat{v}$ satisfies
\begin{equation}\label{fft 1}
\widehat{v}_1\in\R\mbox{~~and~~}{\tt conj}\left(\widehat{v}_j\right)=\widehat{v}_{n-i+1}
\mbox{~~for~~} i=1,\ldots, \left[\dfrac{n+1}{2}\right],
\end{equation}
where ${\tt conj}$ denotes the complex conjugation operator and $\left[(n+1)/2\right]$ 
stands for the integer part of $(n+1)/2$. A vector that satisfies \eqref{fft 1} is said to
be conjugate-even. Moreover, if the $n$-vector $\tilde{v}$ is real and conjugate-even, 
then $v=F_n^H \tilde{v}$ is a real conjugate-even vector, where the superscript $^H$ 
denotes transposition and complex conjugation; see, e.g., \cite[Chapter 15]{He} for many
properties of the DFT, as well as \cite{Rojo}.

Let the operator ${\tt vec}$ denote the vectorization of a tube. We say that the tube
$\bm{a}\in\C^n$ is real conjugate-even, if the vector $a={\tt vec}(\bm{a})$ is real
conjugate-even. By exploiting the symmetry of the DFT, it follows that for real 
third-order tensors $\mathcal{A}\in\R^{\ell\times p\times n}$, we have 
\[
\widehat{\mathcal{A}}^{(1)}\in\R^{\ell\times p}\mbox{~~and~~}
\widehat{\mathcal{A}}^{(i)}=\widehat{\mathcal{A}}^{(n-i+2)},\; i=2,3,\ldots, 
\left[\dfrac{n-1}{2}\right].
\]

Algorithm \ref{alg 1} exploits the symmetry properties of the DFT when evaluating the 
t-product of real tensors. 

\begin{algorithm}[H]
\caption{t-product of real third-order tensors}\label{alg 1}
\textbf{Input:} $\mathcal{A}\in\R^{\ell\times q\times n}$, 
$\mathcal{B}\in\R^{q\times p\times n}$.\\
\textbf{Output:} $\mathcal{C}=\mathcal{A}*\mathcal{B}\in\R^{\ell\times p\times n}$.
\begin{algorithmic}[1]
\STATE Compute $\widehat{\mathcal{A}}={\tt fft}(\mathcal{A},[\,], 3)$ and 
$\widehat{\mathcal{B}}={\tt fft}(\mathcal{B},[\,], 3)$
\FOR {$i=1,2\; \text{to}\;  \left[\dfrac{n+1}{2}\right]$}
\STATE $\widehat{\mathcal{C}}^{(i)}=\widehat{\mathcal{A}}^{(i)}
\widehat{\mathcal{B}}^{(i)}$
\ENDFOR
\FOR {$i=\left[\dfrac{n+1}{2}\right]+1\; \text{to}\;  n$}
\STATE $\widehat{\mathcal{C}}^{(i)}={\tt conj}(\widehat{\mathcal{C}}^{(n-i+2)})$
\ENDFOR
\STATE Compute $\mathcal{C}={\tt ifft}\left(\widehat{\mathcal{C}},[\,], 3\right)$
\end{algorithmic}
\end{algorithm}

We recall the main properties of the t-product that will be used in the sequel. Proofs can
be found in \cite{kilmer2013third}.

\begin{proposition}
\begin{enumerate}[label=(\roman*)]
\item 
The identity tensor under the t-product, $\mathcal{I}_\ell\in\K^{\ell\times\ell}_n$, has
the first frontal slice equal to the identity matrix and the remaining frontal slices are
zero matrices. 
\item
The canonical lateral slice under the t-product, $\vec{\mathcal{E}}_j\in\K^\ell_n$,
where $1\leq j\leq\ell$, has the $(j,1,1)$th entry equal to one, and all other entries
zero. 
\item
The unit tube, $\bm{e}\in\K_n$, has the entry $(1,1,1)$ equal to one, and all other 
entries zero.
\item
The conjugate transpose, $\mathcal{A}^H\in\K^{p\times\ell}_n$, of the tensor 
$\mathcal{A}\in\K^{\ell\times p}_n$ is obtained by first transposing and conjugating each 
one of the frontal slices of $\mathcal{A}$, and then reversing the order of the conjugated
transposed frontal slices 2 through $n$. The tensor conjugate transpose has similar 
properties as the matrix transpose. For instance, if $\mathcal{A}$ and $\mathcal{B}$ are 
tensors such that $\mathcal{A}*\mathcal{B}$ and $\mathcal{B}^H*\mathcal{A}^H$ are defined, 
then $(\mathcal{A}*\mathcal{B})^H = \mathcal{B}^H*\mathcal{A}^H$.
\item
A tensor $\mathcal{A}\in\K^{p\times p}_n$ is said to be normal if 
$\mathcal{A}^H*\mathcal{A}=\mathcal{A}*\mathcal{A}^H$.
\item
A third order tensor is said to be f-diagonal if its frontal slices in the Fourier 
domain are diagonal matrices. 
\item
A third order tensor is said to be f-upper-triangular if its frontal slices in the Fourier 
domain are upper triangular matrices. 
\item
The tensor $\mathcal{A}\in\K^{\ell\times\ell}_n$ is said to be f-Hermitian if 
$\mathcal{A}^H=\mathcal{A}$. In order for $\mathcal{A}$ to be f-Hermitian each frontal
slice of $\widehat{\mathcal{A}}$ has to be a Hermitian matrix.
\item
A tensor $\mathcal{Q}\in\K^{\ell\times\ell}_n$ is said to be f-orthogonal (or f-unitary) 
if $\mathcal{Q}^H*\mathcal{Q}=\mathcal{Q}*\mathcal{Q}^H=\mathcal{I}_\ell$. 
\item
A tensor $\mathcal{A}\in\K^{\ell\times\ell}_n$ is invertible if there is a tensor 
$\mathcal{A}^{-1}\in\K^{\ell\times\ell}_n$ such that
$\mathcal{A}^{-1}*\mathcal{A}=\mathcal{A}*\mathcal{A}^{-1}= \mathcal{I}_\ell$. This 
definition implies that a third order tensor is invertible under the t-product if and only
if each frontal slice in the Fourier domain is invertible. 
\item
Let $p$ be a positive integer and $\mathcal{A}\in\K^{\ell\times\ell}_n$. Then  
$\mathcal{A}^p=\mathcal{A}*\ldots*\mathcal{A}$. We define $\mathcal{A}^0=\mathcal{I}_\ell$.
\item
Let $\mathcal{A}\in\K^{\ell\times p}_n$. Then $\mathcal{A}\not\equiv\textbf{0}$ means that 
$\widehat{\mathcal{A}}^{(i)}\neq 0$ for every $i=1,2,\ldots,n$.
\end{enumerate}
\end{proposition}
\smallskip

In the following we will require the t-product of two tubes. This operation has been 
defined in \cite{kilmer2013third}. The definition suggests how the quotient of two tubes 
can be defined.
\smallskip

\begin{definition}
Let $\bm{a},\bm{b}\in\K_n$. Then 
\[
\bm{a}*\bm{b}={\tt ifft}(\widehat{\bm{a}}\circ\widehat{\bm{b}},[],3),
\]
where $\circ$ denotes the entry-wise product of the vectors and, as usual,
$\widehat{\bm{a}}$ and $\widehat{\bm{b}}$ stand for the Fourier transform of the tubes 
$\bm{a}$ and $\bm{b}$, respectively. Note that the tube product commutes. 

Let, in addition, all entries of $\widehat{\bm{b}}$ be nonvanishing. Then we 
can define 
\[
\frac{\bm{a}}{\bm{b}}={\tt ifft}(\widehat{\bm{a}}\div\widehat{\bm{b}},[],3),
\]
where $\div$ denotes entry-wise division. 
\end{definition}
\smallskip

We also will need to evaluate the product of a tensor and a tube, as well as the
quotient of a tensor and a tube.
\smallskip

\begin{definition}
Let $\mathcal{A}\in\K^{\ell\times p}_n$ and let $\bm{b}\in\K_n$. With the usual notation, 
we define the product $\mathcal{C}=\mathcal{A}*\bm{b}\in\K^{\ell\times p}_n$ as 
\[
\mathcal{C}={\tt ifft}(\widehat{\mathcal{C}},[],3), \mbox{~~where~~}
\widehat{\mathcal{C}}^{(i)}=\widehat{\mathcal{A}}^{(i)}\widehat{\bm{b}}^{(i)},\quad 
i=1,2,\ldots,n.
\]
We note that 
\begin{equation}\label{commute}
\mathcal{A}*\bm{b}=\bm{b}*\mathcal{A}.
\end{equation}
If, in addition, all entries of $\widehat{\bm{b}}$ are nonvanishing, then we define the 
quotient $\mathcal{C}=\mathcal{A}/\bm{b}\in\K^{\ell\times p}_n$ by 
$\mathcal{C}={\tt ifft}(\widehat{\mathcal{C}},[],3)$, where
\[
\widehat{\mathcal{C}}^{(i)}=\widehat{\mathcal{A}}^{(i)}/\widehat{\bm{b}}^{(i)},\quad 
i=1,2,\dots,n.
\]
\end{definition}
\smallskip

We recall that the lateral slices $\vec{\mathcal{A}}_j\in\K^\ell_n$, $j=1,2,\ldots,p$, of 
the tensor  $\mathcal{A}\in\K^{\ell\times p}_n$ are third-order tensors. Therefore, it 
follows from \eqref{commute} that for any tube $\bm{b}\in\K_n$,
\begin{equation}\label{commute2}
\vec{\mathcal{A}}_j*\bm{b}=\bm{b}*\vec{\mathcal{A}}_j.
\end{equation}
We will use this property below.
\smallskip
	
It is convenient to introduce the following space.
\smallskip

\begin{definition}
Let $\mathbb{G}=\{\vec{\mathcal{A}}_1,\vec{\mathcal{A}}_2,\ldots,\vec{\mathcal{A}}_p\} 
\subset\K^\ell_n$ be the set of lateral slices of the tensor 
$\mathcal{A}\in\K^{\ell\times p}_n$. Then the set of the t-linear combinations of
these slices is a lateral slice space given by
\begin{eqnarray*}
\tspan\left\{\mathbb{G}\right\}&=&\tspan\left\{\vec{\mathcal{A}}_1,\vec{\mathcal{A}}_2, 
\ldots,\vec{\mathcal{A}}_p\right\} \nonumber \\
	&=& \left\{\vec{\mathcal{C}}\in\K^\ell_n:~
	\vec{\mathcal{C}}=\sum_{i=1}^{p}\vec{\mathcal{A}}_i*\bm{b}_i,\mbox{~~where~~} 
	\bm{b}_i\in\K^n,~i=1,2,\ldots,p\right\}.
\end{eqnarray*}
\end{definition}

\section{Tensor factorizations}\label{sec 3}
Many matrix factorizations can be extended to third-order tensors under the t-product. We 
will review two available factorizations and introduce a new one, that will be used in the
sequel.

\begin{theorem} (t-SVD) (\cite{kilmer2013third})
Let $\mathcal{A}\in\K^{\ell\times p}_n$. Then, using the t-product, $\mathcal{A}$ 
can be factored into three third-order tensors,
\begin{equation}\label{eq 18}
	\mathcal{A}=\mathcal{U}*\mathcal{S}*\mathcal{V}^H,
\end{equation}
where $\mathcal{U}\in\K^{\ell\times\ell}_n$ and $\mathcal{V}\in\K^{p\times p}_n$ are 
f-orthogonal tensors and $\mathcal{S}\in\K^{\ell\times p}_n$ is an f-diagonal tensor. We 
refer to the factorization \eqref{eq 18} as the \emph{singular tube decomposition}.
\end{theorem}

The tubes $\bm{s}_i=\mathcal{S}_{i,i,:}$ for $i=1,2,\ldots,\min\{\ell,p\}$ of the tensor 
$\mathcal{S}$ are referred to as singular tubes, whereas 
$\vec{\mathcal{U}}_j=\mathcal{U}_{:,j,:}$ for $j=1,2,\ldots,\ell$ and 
$\vec{\mathcal{V}}_k=\mathcal{V}_{:,k,:}$ for $k=1,2,\ldots, p$ are called left and right 
singular (lateral) slices, respectively. The norm of the singular tubes $\bm{s}_i$, given 
by $\sigma_i=\left\Vert\bm{s}_i\right\Vert_F$ for $i=1,2,\ldots,\min\{\ell,p\}$, are 
referred to as the singular values of the tensor $\mathcal{A}$; they are enumerated in 
decreasing order. The singular triplets 
$\{\bm{s}_i, \vec{\mathcal{U}}_i, \vec{\mathcal{V}}_i\}_{i=1}^{\min\{\ell,p\}}$ satisfy
\[
\mathcal{A}*\vec{\mathcal{V}}_i=\vec{\mathcal{U}}_i*\bm{s}_i,~~ 
\mathcal{A}^H*\vec{\mathcal{U}}_i=\vec{\mathcal{V}}_i*\bm{s}_i,~~
1=1,2,\ldots,\min\{\ell,p\}.
\]
We note that equation \eqref{eq 18} can be written as 
\[
\mathcal{A}=\sum_{i=1}^{\min\{\ell,p\}}\vec{\mathcal{U}}_i*\bm{s}_i*\vec{\mathcal{V}}_i^H=
\sum_{i=1}^{\min\{\ell,p\}}\bm{s}_i*\vec{\mathcal{U}}_i*\vec{\mathcal{V}}_i^H,
\]
where the last equality follows from \eqref{commute2}.

Kilmer et al. \cite{kilmer2013third} also describe the t-QR factorization of a tensor
$\mathcal{A}\in\K^{\ell\times p}_n$,
\[
\mathcal{A}=\mathcal{Q}*\mathcal{R},
\]
where $\mathcal{Q}\in\K^{\ell\times\ell}_n$ is an f-orthogonal tensor and 
$\mathcal{R}\in\K^{\ell\times p}_n$ is an f-upper-triangular tensor.

\begin{theorem}
Let the third-order tensor $\mathcal{A}\in\K^{p\times p}_n$ be invertible. Then LU 
factorization with partial pivoting can be applied in the Fourier domain, resulting in
the t-LU decomposition
\begin{equation}\label{t-LU}
\mathcal{P}*\mathcal{A}=\mathcal{L}*\mathcal{U},
\end{equation}
where $\mathcal{P}\in\K^{p\times p}_n$ is the Fourier transform of permutation matrices, 
$\mathcal{L}\in\K^{p\times p}_n$ is an f-unit-lower triangular tensor, and 
$\mathcal{U}\in\K^{p\times p}_n$ is an f-upper-triangular tensor.
\end{theorem}

\begin{proof}
Let $\widehat{\mathcal{A}}={\tt fft}(\mathcal{A},[\,], 3)$. Then, for $i=1,\ldots, n$, we 
apply LU factorization with partial pivoting to the frontal slices
$\widehat{\mathcal{A}}^{(i)}$ of $\widehat{\mathcal{A}}$. This gives the decompositions
\[
\widehat{P}^{(i)}\widehat{A}^{(i)}=\widehat{L}^{(i)}\widehat{U}^{(i)},\quad 
i=1,2,\ldots,n,
\]
where $\widehat{P}^{(i)}$ is a permutation matrix, $\widehat{L}^{(i)}$ is a unit lower 
triangular matrix, and $\widehat{U}^{(i)}$ is an upper triangular matrix. These 
decompositions exist because the tensor $\mathcal{A}$ is nonsingular. The matrices
$\widehat{P}^{(i)}$, $\widehat{L}^{(i)}$, and $\widehat{U}^{(i)}$ are frontal slices of
the tensors $\widehat{P}$, $\widehat{L}$, and $\widehat{U}$. The tensors in the 
decomposition \eqref{t-LU} are determined by 
\[
\mathcal{P}={\tt ifft}\left(\widehat{\mathcal{P}},[\,],3\right),\quad 
\mathcal{L}={\tt ifft}\left(\widehat{\mathcal{L}},[\,],3\right),\quad
\mathcal{U}={\tt ifft}\left(\widehat{\mathcal{U}},[\,],3\right).
\]
\end{proof}

The following factorization generalizes the Hessenberg reduction for matrices. It will be
used below in the computation of tensor eigentubes and eigenslices. 

\begin{theorem}
The tensor $\mathcal{A}\in\K^{p\times p}_n$ is unitarily similar to an f-upper-Hessenberg 
tensor, i.e.,
\[
\mathcal{H}=\mathcal{W}^H *\mathcal{A}*\mathcal{W},
\]
where $\mathcal{H}\in\mathbb{K}^{p\times p}_n$ is an f-upper-Hessenberg tensor (all 
frontal slices in the Fourier domain are upper Hessenberg matrices) and 
$\mathcal{W}\in \mathbb{K}^{p\times p}_n$ is an f-unitary tensor.
\end{theorem}

\begin{proof}
Let $\widehat{\mathcal{A}}={\tt fft}(\mathcal{A},[\,], 3)$. Then, for $i=1,\ldots, n$, the
frontal slice $\widehat{\mathcal{A}}^{(i)}$ of $\widehat{\mathcal{A}}$ is unitarily 
similar to a Hessenberg matrix $H_i$. We have 
\[
H_i=W_i^H \widehat{\mathcal{A}}^{(i)} W_i,\quad i=1,2,\ldots,n,
\]
where the $H_i\in\C^{p\times p}$ are upper Hessenberg matrices and the 
$W_i\in\C^{p\times p}$ are unitary matrices.

Define the tensors $\mathcal{H}'$ and $\mathcal{W}'$ as follows:
\[
\mathcal{H}'^{(i)}=H_i {\rm ~~and~~} \mathcal{W}'^{(i)}=W_i,\quad i=1,2,\ldots,n.
\]
Then 
\[
\mathcal{H}'^{(i)}=\left(\mathcal{W}'^{(i)}\right)^H \widehat{\mathcal{A}}^{(i)} 
\mathcal{W}'^{(i)},\quad i=1,2,\ldots,n.
\]
Consequently, defining the tensors 
$\mathcal{H}={\tt ifft}\left(\mathcal{H}',[\,],3\right)$ and 
$\mathcal{W}={\tt ifft}\left(\mathcal{W}',[\,],3\right)$, we  conclude that
\[
\mathcal{H}=\mathcal{W}^H *\mathcal{A}*\mathcal{W},
\]
where $\mathcal{H}$ is an f-Hessenberg tensor since each frontal slice is an upper
Hessenberg matrix in the Fourier domain; obviously the tensor $\mathcal{W}$ is unitary, 
since its frontal slices in the Fourier domain are unitary.
\end{proof}

Algorithm \ref{alg 2} provides details of the transformation of a tensor to 
f-upper-Hessenberg form.
\begin{algorithm}[H]
\caption{t-Hessenberg reduction ({\tt t-hess})}\label{alg 2}
\textbf{Input:} $\mathcal{A}\in \mathbb{K}^{p\times p}_n$.\\
\textbf{Output:} Hessenberg tensor $\mathcal{H}\in\K^{p\times p}_n$ and unitary tensor 
$\mathcal{W}\in\K^{p\times p}_n$.\hfill\vskip.1in
\begin{algorithmic}[1]
\STATE $\widehat{\mathcal{A}}={\tt fft}\left(\mathcal{A},[\,],3\right)$.
\FOR {$i=1$ to $n$}
\STATE $[\widehat{\mathcal{W}}^{(i)}, 
\widehat{\mathcal{H}}^{(i)}]={\tt hess}(\mathcal{A}^{(i)})$
\ENDFOR
\STATE $\mathcal{W}={\tt ifft}\left(\widehat{\mathcal{W}},[\,],3\right)$, 
$\mathcal{H}={\tt ifft}\left(\widehat{\mathcal{H}},[\,],3\right)$
\end{algorithmic}
\end{algorithm}

The MATLAB function ${\tt hess}$ used in line $3$ of Algorithm \ref{alg 2} applies a unitary
similarity transformation $\widehat{\mathcal{W}}^{(i)}$ to the matrix 
$\widehat{\mathcal{A}}^{(i)}$ to determine an upper Hessenberg matrix 
$\widehat{\mathcal{H}}^{(i)}$. In the case of real tensors we can use the symmetry 
property of the DFT.

Following Kilmer et al. \cite{kilmer2013third}, we introduce a bilinear form associated 
with the t-product. For lateral slices 
$\vec{\mathcal{X}},\vec{\mathcal{Y}}\in\K^\ell_n$, we define
\[
\langle \vec{\mathcal{X}},\vec{\mathcal{Y}}\rangle=
\vec{\mathcal{X}}^H *\vec{\mathcal{Y}}\in\K_n.
\]
The value of this bilinear form is a tube and the quantity 
$\langle\vec{\mathcal{X}},\vec{\mathcal{X}}\rangle$ is not necessarily real and positive. 
The  next lemma gives some properties of this bilinear form.

\begin{lemma}(\cite{kilmer2013third})
Let $\vec{\mathcal{X}},\vec{\mathcal{Y}},\vec{\mathcal{Z}}\in\K^p_n$ and $\bm{a}\in\K_n$. 
Then the following properties hold:
\begin{enumerate}
\item $\langle  \vec{\mathcal{X}}, \vec{\mathcal{Y}}+ \vec{\mathcal{Z}}\rangle= 
\langle \vec{\mathcal{X}},\vec{\mathcal{Y}} \rangle+ 
\langle \vec{\mathcal{X}},\vec{\mathcal{Z}} \rangle$,
\item $\langle \vec{\mathcal{X}},\vec{\mathcal{Y}}*\bm{a}\rangle=
\bm{a}*\langle\vec{\mathcal{X}},\vec{\mathcal{Y}}\rangle$,
\item $\langle \bm{a}*\vec{\mathcal{X}},\vec{\mathcal{Y}} \rangle=
\bm{a}^H *\langle\vec{\mathcal{X}},\vec{\mathcal{Y}}\rangle$,
\item $\langle \vec{\mathcal{X}}.\vec{\mathcal{Y}}  \rangle=
\langle \vec{\mathcal{Y}}, \vec{\mathcal{X}} \rangle^H$.
\end{enumerate}
\end{lemma}

A sequence $\left\{\vec{\mathcal{X}}_1,\vec{\mathcal{X}}_2,\ldots,\vec{\mathcal{X}}_p
\right\}\subset\K^\ell_n$ of $p\geq 2$ lateral slices is said to be f-orthogonal if  
\[
\langle\vec{\mathcal{X}}_i,\vec{\mathcal{X}}_j\rangle=
\vec{\mathcal{X}}_i^H *\vec{\mathcal{X}}_j=\left\lbrace\begin{array}{ll}
\alpha_i \textbf{e} & \mbox{if $i=j$},\;\; i,j=1,2 , \ldots, p,\\
\textbf{0} & \mbox{else},
\end{array}
\right.
\]
where $\bm{e}$ is the unit tube and $\alpha_i\in\C$. If $\alpha_i=1$ for all 
$1\leq i\leq p$, then the sequence is said to be f-orthonormal.

Let the lateral slice $\vec{\mathcal{Y}}\in\K^\ell_n$ be nonvanishing. We then can 
normalize $\vec{\mathcal{Y}}$ in order to obtain a slice $\vec{\mathcal{X}}\in\K^\ell_n$ 
with norm $\Vert\vec{\mathcal{X}}\Vert=1$, i.e., 
$\vec{\mathcal{Y}}=\vec{\mathcal{X}}*\bm{a}$ with
$\Vert\vec{\mathcal{Y}}\Vert=\Vert\bm{a}\Vert$, where 
\[
\Vert\vec{\mathcal{Y}}\Vert=
\dfrac{\Vert\langle\vec{\mathcal{Y}},\vec{\mathcal{Y}}\rangle\Vert_F}
{\Vert\vec{\mathcal{Y}}\Vert_F}.
\]
\section{Tensor eigenpairs based on the t-product}\label{sec 4}
This section introduces tensor eigenpairs for third-order tensors defined with the aid of 
the t-product. 

\begin{definition}
Let $\mathcal{A}\in\K^{p\times p}_n$ and let 
$\gamma_{1,i},\gamma_{2,i},\ldots,\gamma_{p,i}$ be the eigenvalues of 
$\widehat{\mathcal{A}}^{(i)}$ for $i=1,2,\ldots,n$, ordered such as 
\[
\vert\gamma_{j,i}\vert \geq \vert\gamma_{j+1,i}\vert,\quad j=1,2,\ldots,p-1.
\]
The tubes $\bm{\lambda}_1,\bm{\lambda}_2,\ldots,\bm{\lambda}_p$ defined by 
\[
\widehat{\bm{\lambda}}_j^{(i)}=\gamma_{j,i},\; i=1,2,\ldots,n,\; j=1,2,\ldots,p,
\]
are \emph{eigentubes} of $\mathcal{A}$. For a given eigentube $\bm{\lambda}$ of 
$\mathcal{A}$, a nonvanishing lateral slice $\vec{\mathcal{U}}\in \K^p_n$ such that
\begin{equation}\label{eslice}
\mathcal{A}*\vec{\mathcal{U}}=\bm{\lambda}*\vec{\mathcal{U}}
\end{equation}
is said to be an \emph{eigenslice} or \emph{right eigenslice} of $\mathcal{A}$ associated 
with $\bm{\lambda}$. Moreover, the set $\{\bm{\lambda},\vec{\mathcal{U}}\}$ is referred to
as an eigenpair of $\mathcal{A}$. 
\end{definition}
\smallskip

Note that equations \eqref{eslice} and \eqref{commute2}, imply that
\[
\mathcal{A}*\vec{\mathcal{U}}=\vec{\mathcal{U}}*\bm{\lambda}.
\]
We remark that when $n=1$, the third-order tensor reduces to a square matrix, and the 
eigentubes and eigenslices become eigenvalues and eigenvectors, respectively.

A tensor $\mathcal{A}\in\K_n^{p\times p}$ has $p$ eigentubes. Let  
$\bm{\lambda}_1,\bm{\lambda}_2,\ldots,\bm{\lambda}_p$ be the eigentubes of $\mathcal{A}$.
Then 
\[
\left\Vert \bm{\lambda}_1 \right\Vert_F \geq \left\Vert \bm{\lambda}_2 \right\Vert_F\geq
\ldots \geq \left\Vert \bm{\lambda}_p \right\Vert_F.
\]
The set of all eigentubes of $\mathcal{A}$ is referred to as the \emph{spectrum} of 
$\mathcal{A}$ and is denoted by $\bm{\Lambda}\left(\mathcal{A}\right)$. The 
\emph{spectral radius} of $\mathcal{A}$ is defined as 
\[
\rho\left(\mathcal{A}\right) = 
\max_{\bm{\lambda}\in \bm{\Lambda}\left(\mathcal{A}\right)}
\left\Vert \bm{\lambda}\right\Vert_F.
\]

The eigenvalues of a matrix are the zeros of the characteristic polynomial for the matrix.
We will derive an analogous result for third-order tensors. To this end, we require the
definition of the determinant of a square third-order tensor.
\smallskip

\begin{definition}
Let $\mathcal{A}\in\mathbb{K}^{p\times p}_n$. The determinant of $\mathcal{A}$ under the 
t-product is the tube 
\[
\bm{\gamma}=\tdet\left(\mathcal{A}\right)\in \mathbb{K}_n,
\]
which is defined by the inverse Fourier transformation of the tube $\widehat{\bm{\gamma}}$
with $\widehat{\bm{\gamma}}^{(i)}=\det\left(\widehat{\mathcal{A}}^{(i)}\right)$, 
$i=1,2,\ldots,n$, where $\widehat{\mathcal{A}}={\tt fft}(\mathcal{A},[\,],3)$.
\end{definition}
\smallskip

Thus, the determinant of a third-order tensor under the t-product is a function
\[
\begin{split}
	\tdet: \mathbb{K}^{p\times p}_n &\to \mathbb{K}_n \\
	\mathcal{A} &\mapsto \tdet \left(\mathcal{A}\right),
\end{split}
\]
such that $\widehat{\tdet\left(\mathcal{A}\right)}^{(i)}=
\det\left(\widehat{\mathcal{A}}^{(i)}\right)$ for $i=1,2,\ldots,n$.
This determinant has the following properties.

\begin{proposition}\label{prop 3}
Let $\mathcal{A},\mathcal{B}\in\mathbb{K}^{p\times p}_n$ and let 
$\bm{\alpha}\in\mathbb{K}_n$. Then
\begin{enumerate}[label=(\roman*)]
\item 
$\tdet\left(\mathcal{A}*\mathcal{B}\right)=\tdet\left(\mathcal{B}*\mathcal{A}\right)$,
\item 
$\tdet\left(\mathcal{A}^H\right)=\left(\tdet\left(\mathcal{A}\right)\right)^H$,
\item 
$\tdet\left(\bm{\alpha}*\mathcal{A}\right)=\bm{\alpha}^p*\tdet\left(\mathcal{A}\right)$.
\end{enumerate}
\end{proposition}

\begin{proof}
\begin{enumerate}[label=(\roman*)]
\item 
Let $\mathcal{C}=\mathcal{A}*\mathcal{B}$. Then 
$\widehat{\mathcal{C}}^{(i)}=\widehat{\mathcal{A}}^{(i)}\widehat{\mathcal{B}}^{(i)}$ for 
$i=1,2,\dots,n$. Consider $\bm{\gamma}=\tdet\left(\mathcal{A}*\mathcal{B}\right)$. Then 
\begin{eqnarray*}
\widehat{\bm{\gamma}}^{(i)}=\det\left(\widehat{\mathcal{C}}^{(i)}\right)&=& 
\det\left(\widehat{\mathcal{A}}^{(i)}\widehat{\mathcal{B}}^{(i)}\right)\\
	&=&\det\left(\widehat{\mathcal{B}}^{(i)}\widehat{\mathcal{A}}^{(i)}\right).
\end{eqnarray*}
We conclude that $\tdet\left(\mathcal{A}*\mathcal{B}\right)=
\tdet\left(\mathcal{B}*\mathcal{A}\right)$.
\item 
Let $\bm{\gamma}=\tdet\left(\mathcal{A}^H\right)$. Then
\begin{eqnarray*}
\widehat{\bm{\gamma}}^{(i)}=\det\left(\widehat{\mathcal{A}^H}^{(i)}\right)=
\det\left(\left(\widehat{\mathcal{A}}^{(i)}\right)^H\right)=
{\tt conj}\left(\det\left(\widehat{\mathcal{A}}^{(i)}\right)\right).
\end{eqnarray*}
This shows that $\tdet\left(\mathcal{A}^H\right)=
\left(\tdet\left(\mathcal{A}\right) \right)^H$.
\item 
Let $\mathcal{A}\in\mathbb{K}^{p\times p}_n$ and $\bm{\alpha}\in\mathbb{K}_n$. Then
\[
\det\left(\widehat{\bm{\alpha}}^{(i)}\widehat{\mathcal{A}}^{(i)}\right)=
\left(\widehat{\bm{\alpha}}^{(i)}\right)^p \det\left( \widehat{\mathcal{A}}^{(i)}\right).
\]
By the definition of the power of a third-order tensor, we obtain
$\tdet\left(\bm{\alpha}*\mathcal{A}\right)=\bm{\alpha}^p*\tdet\left(\mathcal{A}\right)$.
\end{enumerate}
\end{proof}

Define the polynomial $\bm{P}$ of degree $s$ on the commutative ring 
$\left(\mathbb{K}_n, +, *\right)$ as follows:
\[
\bm{P}\left(\bm{x}\right)=\sum_{i=0}^{s}\bm{a}_i*\bm{x}^i,
\]
where $\bm{a}_i\in\mathbb{K}_n$ for $i=0,1,\ldots,s$ and $\bm{a}_s\not\equiv 0$. We remark
that if $\bm{P}$ is a polynomial in $\mathbb{K}_n$ of degree $s$, then each polynomial 
$\bm{P}^{(i)}$ defined by
\[
\bm{P}^{(i)}(\bm{x})=\sum_{j=0}^{s}
\widehat{\bm{a}}_j^{(i)}\left(\widehat{\bm{x}}^{(i)}\right)^j,\quad i=1,2,\ldots,n, 
\]
is of degree $s$.

\begin{definition}\label{defroots}
Let $\bm{P}$ be a polynomial in $\K_n$ of degree $s$. We say that 
$\bm{\gamma}_1,\bm{\gamma}_2,\ldots,\bm{\gamma}_s$ are roots of $\bm{P}$, i.e., they are
solutions of $\bm{P}(\bm{x})=0$, if 
$\widehat{\bm{\gamma}}_1^{(i)},\ldots,\widehat{\bm{\gamma}}_s^{(i)}$ are roots of 
$\bm{P}^{(i)}$. We order the roots according to 
$\vert\widehat{\bm{\gamma}}_j^{(i)}\vert\geq\vert\widehat{\bm{\gamma}}_{j+1}^{(i)}\vert$ 
for $j=1,2,\ldots,s-1$.
\end{definition}
\smallskip

It follows from Definition \ref{defroots} that every polynomial in $\K_n$ of degree $s$ 
has $s$ roots in $\K_n$.
\smallskip

\begin{definition}
Let $\bm{P}$ be a polynomial of degree $s$ in $\mathbb{K}_n$. We say that a root 
$\bm{\gamma}$ of $\bm{P}$ is of f-multiplicity
\[
m=\min_{1\leq i\leq n}\, {\rm mult}\left(\widehat{\bm{\gamma}}^{(i)}\right),
\]
where for each $i=1,2,\ldots,n$, ${\rm mult}\left(\widehat{\bm{\gamma}}^{(i)}\right)$ 
denotes the multiplicity of the root $\widehat{\bm{\gamma}}^{(i)}$.
\end{definition}
\smallskip

\begin{definition}
Let $\mathcal{A}\in \K^{p\times p}_n$. Then the characteristic polynomial of $\mathcal{A}$
is given by 
\[
\bm{P}_\mathcal{A}\left(\bm{x}\right)=\tdet\left(\mathcal{A}-\mathcal{I}_p*\bm{x}\right),
\quad \bm{x}\in\K_n.
\]
\end{definition}
\smallskip

We remark that the degree of the characteristic polynomial of the tensor 
$\mathcal{A}\in\K^{p\times p}_n$ is $p$ since the degree of the characteristic 
polynomial of each frontal slice of $\widehat{\mathcal{A}}$ is $p$. Furthermore, the 
characteristic polynomial of $\mathcal{A}$ has $p$ roots.

\begin{proposition}
An eigentube of the tensor $\mathcal{A}\in\K^{p\times p}_n$ is a root of the 
characteristic polynomial $\bm{P}_\mathcal{A}$ of $\mathcal{A}$. 
\end{proposition}

\begin{proof}
Let $\bm{\lambda}$ be an eigentube of $\mathcal{A}$. Then $\widehat{\bm{\lambda}}^{(i)}$ 
is an eigenvalue of $\widehat{\mathcal{A}}^{(i)}$ for $i=1,2,\ldots,n$. Moreover, 
$\widehat{\bm{\lambda}}^{(i)}$ is a root of the characteristic polynomial of 
$\widehat{\mathcal{A}}^{(i)}$ for $i=1,2,\ldots,n$. Therefore 
$\bm{P}_\mathcal{A}\left(\bm{\lambda}\right)=\tdet\left(\mathcal{A}-\mathcal{I}_p* 
\bm{\lambda}\right)=\bm{0}$.
\end{proof}
\smallskip

Note that the tensor $\mathcal{A}\in\K^{p\times p}_n$ has at most $p$ distinct eigentubes.

\begin{proposition}
Let $\bm{\lambda}$ be an eigentube of $\mathcal{A}\in\K^{p\times p}_n$. Then 
$\bm{\lambda}^H$ is an eigentube of $\mathcal{A}^H$.
\end{proposition} 

\begin{proof}
This result follows from the observation that an eigentube is a root of the 
characteristic polynomial of $\mathcal{A}$, and by using Proposition \ref{prop 3}.
\end{proof}

Let $\bm{\lambda}$ be an eigentube of $\mathcal{A}\in\K^{p\times p}_n$. The slice 
$\vec{\mathcal{V}}\in\K^{p}_n$ such that 
\[
\mathcal{A}^H * \vec{\mathcal{V}}=\bm{\lambda}^H*\vec{\mathcal{V}}
\]
is called a \emph{left eigenslice} of $\mathcal{A}$ associated with $\bm{\lambda}$. 

We turn to the linear independence of lateral slices. Let 
$\{\vec{\mathcal{V}}_i\}_{i=1}^s\subset \K^{p}_n$ be a set of lateral slices. This set is 
t-linearly independent if and only if 
\[
\sum_{i=1}^{s}\bm{a}_i *\vec{\mathcal{V}}_i=\bm{0}\Longrightarrow \bm{a}_i = \bm{0},\quad 
i=1,2,\ldots,s.
\]
It is clear that the slices $\vec{\mathcal{V}}_1,\ldots,\vec{\mathcal{V}}_s$ in $\K_n^p$ 
are t-linearly independent if and only if the vectors 
$ \vec{\widehat{\mathcal{V}}}_1^{(j)},\ldots,\vec{\widehat{\mathcal{V}}}_s^{(j)}$ are 
linearly independent for $j=1,2,\ldots,n$. This leads to the following result.

\begin{proposition}\label{prop 9}
A tensor $\mathcal{A}\in\K^{p\times p}_n$ is invertible under the t-product if it has $p$ 
t-linearly independent lateral slices.
\end{proposition}

We are in a position to define f-diagonalization of a third-order tensor.

\begin{definition}
A third-order tensor $\mathcal{A}\in \K^{p\times p}_n$ is said to be f-diagonalizable 
under the t-product if and only if it is similar to an f-diagonal tensor, i.e., if and
only if 
\[
\mathcal{A}=\mathcal{X}* \mathcal{D}* \mathcal{X}^{-1}
\]
for some invertible tensor $\mathcal{X}\in \K^{p\times p}_n$ and an f-diagonal tensor $\mathcal{D}$ 
formed by the eigentubes.
\end{definition}

\begin{proposition}
A tensor $\mathcal{A}\in \K^{p\times p}_n$ is f-diagonalizable if and only if it has $p$ 
linearly independent eigenslices.
\end{proposition}

\begin{proposition}
Let $\vec{\mathcal{U}}$ be an eigenslice of $\mathcal{A}^H*\mathcal{A}$ associated with 
the eigentube $\bm{\lambda}$. Then $\mathcal{A}*\vec{\mathcal{U}}$ is an eigentube of 
$\mathcal{A}*\mathcal{A}^H$ associated to $\bm{\lambda}$.
\end{proposition}

\begin{proof}
We have $\mathcal{A}^H *\mathcal{A}*\vec{\mathcal{U}}=\vec{\mathcal{U}}*\bm{\lambda}$.
Therefore, $\mathcal{A}*\mathcal{A}^H *\mathcal{A}*\vec{\mathcal{U}}=
\mathcal{A}*\vec{\mathcal{U}}*\bm{\lambda}$.
\end{proof}

\begin{theorem}
Let the tensor $\mathcal{A}\in \mathbb{K}^{p\times p}_n$ be f-Hermitian. Then all
eigentubes $\bm{\lambda}$ of $\mathcal{A}$ are real.
\end{theorem}

\begin{proof}
Let $\{\bm{\lambda},\vec{\mathcal{U}}\}$ be an eigenpair of $\mathcal{A}$. Then   
\[
\bm{\lambda}=\dfrac{\langle \vec{\mathcal{U}},\mathcal{A}*\vec{\mathcal{U}}\rangle}
{\langle \vec{\mathcal{U}},\vec{\mathcal{U}}\rangle}=
\dfrac{\langle\mathcal{A}^H*\vec{\mathcal{U}},\vec{\mathcal{U}}\rangle}
{\langle \vec{\mathcal{U}},\vec{\mathcal{U}}\rangle}=\bm{\lambda}^H.
	\]
Since $\mathcal{A}$ is f-Hermitian, we have that each frontal slice of 
$\widehat{\mathcal{A}}$ is a Hermitian matrix. Thus the eigenvalues of each frontal slice
of $\widehat{\mathcal{A}}$ are real. Each eigentube of $\mathcal{A}$ is given by 
$\bm{\lambda}={\tt ifft}(\widehat{\bm{\lambda}},[\,], 3)$, where $\widehat{\bm{\lambda}}$
is a real conjugate-even tube. Consequently, $\bm{\lambda}$ is a real conjugate-even tube.
\end{proof}

Similarly as for matrices, we can relate eigentubes to tensor singular values. Let 
$\mathcal{A}\in\K^{n\times m}_p$ be a third-order tensor with tubal rank $r>0$. Consider
the analogue of the singular tube decomposition \eqref{eq 18} of $\mathcal{A}$ with the 
singular tubes are ordered according to
\[
\left\Vert \bm{s}_1 \right\Vert_F \geq \left\Vert \bm{s}_2 \right\Vert_F \geq \ldots 
\geq \left\Vert \bm{s}_r \right\Vert_F \geq \Vert \bm{s}_{r+1}\Vert_F=\ldots=
\Vert\bm{s}_{\min\{m,n\}}\Vert_F=0.
\]
Then
\[
\begin{array}{rcll}
\mathcal{A}*\vec{\mathcal{V}}_i&=&\vec{\mathcal{U}}_i*\bm{s}_i,\; i=1,2,\ldots,r, \; \;&
\mathcal{A}*\vec{\mathcal{V}}_i=0 , \; i=r+1,r+2,\ldots, m,\\
\mathcal{A}^H *\vec{\mathcal{U}}_i&=&\vec{\mathcal{V}}_i*\bm{s}_i,\; i=1,2,\ldots,r, 
\; \;& \mathcal{A}^H *\vec{\mathcal{U}}_i=0 , \; i=r+1,r+2,\ldots, n.
\end{array}
\]
This implies that $\vec{\mathcal{V}}_1,\ldots,\vec{\mathcal{V}}_m$ are eigenslices of 
$\mathcal{A}^H *\mathcal{A}$, $\vec{\mathcal{U}}_1,\ldots,\vec{\mathcal{U}}_m$ are 
eigenslices of $\mathcal{A}*\mathcal{A}^H$, and $\bm{s}_i^2$, $i=1,2,\ldots,r$, are 
nonzero eigentubes of $\mathcal{A}^H *\mathcal{A}$ and $\mathcal{A}*\mathcal{A}^H$.

Numerical methods for spectral factorization of a real nonsymmetric matrix use the real 
Schur form. We therefore are interested in the analogous factorization of third-order 
tensors. 

\begin{theorem}[Real t-Schur decomposition]\label{theo real-schur}
Let $\mathcal{A}\in \mathbb{R}^{p\times p\times n}$ be a real third-order tensor. There is
a real f-orthogonal tensor $\mathcal{Q}\in\R^{p\times p\times n}$ such that 
\[
\mathcal{Q}*\mathcal{A}*\mathcal{Q}^H=\mathcal{R}=\begin{pmatrix}
\mathcal{R}_{1,1} &  \mathcal{R}_{1,2} & \dots & \dots& \mathcal{R}_{1,m}\\
	0 & \mathcal{R}_{2,2} &  \mathcal{R}_{2,2} & \dots & \mathcal{R}_{2,m}\\
	0 & \dots & \ddots & \dots & \vdots\\
	\vdots & \dots &  & & \vdots\\
	0 & \vdots & & & \mathcal{R}_{m,m}
\end{pmatrix}\in\R^{p\times p \times n},
\]
where the $\mathcal{R}_{i,i}$ are tensors of size $1\times 1 \times n$ or 
$2\times 2 \times n$. The tensor $\mathcal{R}$ is \emph{f-quasi-triangular}.
\end{theorem}

\begin{proof}
The factorization follows from the real Schur factorization of the matrices in each 
frontal slice of the tensor $\widehat{\mathcal{A}}$, and that the first frontal slice of 
$\widehat{\mathcal{A}}$ is real, because 
$\widehat{\mathcal{A}}^{(1)}=\dfrac{1}{\sqrt{n}}\mathcal{A}^{(i)}$.
\end{proof}

For $\mathcal{A}\in\K^{p\times p}_n$, let
\[
r_i=\dim\left(\Null\left(\widehat{\mathcal{A}}^{(i)}\right)\right),\;\; i=1,2,\ldots,n,
\quad r=\displaystyle{\min_{1\leq i\leq n}}r_i.
\]

\begin{definition}\label{defnull}
Let $\mathcal{A}\in\K^{\ell\times p}_n$. The range of $\mathcal{A}$ is the t-span of the 
$p$ lateral slices of $\mathcal{A}$, i.e.,
\[
\Range\left(\mathcal{A}\right)=\left\{\mathcal{A}*\vec{\mathcal{X}}:\,
\vec{\mathcal{X}}\in\K^{p}_n\right\},
\]
and the null space of $\mathcal{A}$ is defined as
\begin{align*}
\tNull \left(\mathcal{A}\right)  = &\left\{\vec{\mathcal{X}}_1,\ldots, \vec{\mathcal{X}}_r
\in\K_n^p:\,\widehat{\vec{\mathcal{X}}}_j^{(i)}\in
\Null\left(\widehat{\mathcal{A}}^{(i)}\right)\; \text{for}~ i=1,2,\ldots,n, \right. \\
 &\left.\quad \text{with}\; \Vert\widehat{\vec{\mathcal{X}}}_j^{(i)}\Vert_F 
 \geq\Vert\widehat{\vec{\mathcal{X}}}_{j+1}^{(i)}\Vert_F,\, j=1,2,\ldots,r_i 
 \right\}.
\end{align*}
\end{definition}

It is clear from Definition \ref{defnull} that if $\vec{\mathcal{X}}\in\K^p_n$ is in the 
null space of a tensor $\mathcal{A}\in\K_n^{p\times p}$, then 
$\mathcal{A}*\vec{\mathcal{X}}=0$.

We introduce some definitions about eigentubes.
\smallskip

\begin{definition}
Let $\mathcal{A}\in\K^{p\times p}_n$.
\begin{itemize}
\item 
An eigentube $\bm{\lambda}$ of $\mathcal{A}$ is said to have algebraic f-multiplicity $m$
if it is a root of f-multiplicity $m$ of the characteristic polynomial.
\item 
Let $\bm{\lambda}$ be an eigentube of $\mathcal{A}$, and let $m_i$ be the geometric 
multiplicity of $\widehat{\bm{\lambda}}^{(i)}$ for $i=1,2,\ldots,n$. Then we refer to
$m=\min_{1\leq i\leq n}m_i$ as the geometric f-multiplicity of 
$\bm{\lambda}$.
\item 
An eigentube $\bm{\lambda}$ is said to be f-simple (f-semi-simple) if for each 
$1\leq i\leq n$, the eigenvalue $\widehat{\bm{\lambda}}^{(i)}$ of 
$\widehat{\mathcal{A}}^{(i)}$ is simple (semi-simple).
\end{itemize}		
\end{definition}
\smallskip

\begin{proposition}
Let $\bm{\lambda}$ be an eigentube of $\mathcal{A}\in\K^{p\times p}_n$ with geometric 
f-multiplicity $m$. Then 
$m=\tdim\left(\tNull\left(\mathcal{A}-\bm{\lambda}*\mathcal{I}_p\right)\right)$.
\end{proposition}

\begin{proof}
Let $\mathcal{A}\in \K^{p\times p}_n$ and 
\[
m_i=\dim\left(\Null\left(\widehat{\mathcal{A}}^{(i)}-
\widehat{\bm{\lambda}}^{(i)}I\right)\right),\quad i=1,2,\ldots,n.
\]
Since $\bm{\lambda}$ has geometric multiplicity $m$, we have $m=\min_{1\leq i\leq n}m_i$,
and the result follows.
\end{proof}
\smallskip
	
\begin{definition}
Assume that the tensor $\mathcal{A}\in \K^{p\times p}_n$ has $\ell$ distinct eigentubes 
$\bm{\lambda}_1,\bm{\lambda}_2,\ldots,\bm{\lambda}_\ell$. The index of $\bm{\lambda}_i$ is
the smallest integer such that 
\[
\tNull\left(\mathcal{A}-\bm{\lambda}_i * \mathcal{I}_p \right)^{\ell_i}=
\tNull\left(  \mathcal{A}-\bm{\lambda}_i * \mathcal{I}_p \right)^{k}, \;\; \forall 
k\geq \ell_i.
\]
\end{definition}
\smallskip

Let $\ell_i$ be the index of the eigentube $\bm{\lambda}_i$ of $\mathcal{A}$. Then 
$\ell_i$ is the maximum index of the eigenvalues $\widehat{\bm{\lambda}}_i^{(j)}$ of the
matrix $\widehat{\mathcal{A}}^{(j)}$ for $j=1,2,\ldots,n$. Note that if $\bm{\lambda}_i$ 
is an f-semi-simple eigentube of $\mathcal{A}$, then $\ell_i=1$.

\begin{proposition}\label{propsum}
Let $\vec{\mathcal{X}}\in\K^{p}_n$. Then $\vec{\mathcal{X}}$ can be represented as
\[
\vec{\mathcal{X}}=\sum_{i=1}^{\ell}\vec{\mathcal{X}}_i, 
\]
where $\vec{\mathcal{X}}_i\in\tNull\left(\mathcal{A}-
\bm{\lambda}_i*\mathcal{I}_p\right)^{\ell_i}$ and $\ell_i$ is the index of 
$\bm{\lambda}_i$ for $i=1,\ldots,\ell$.
\end{proposition}

\begin{proof}
Let $\vec{\mathcal{X}}\in\K^{p}_n$ and let $\widehat{\vec{\mathcal{X}}}$ denote the 
associated tensor in the Fourier domain. The frontal slices 
$\widehat{\vec{\mathcal{X}}}^{(k)}$, $k=1,2,\ldots,n$, of $\widehat{\vec{\mathcal{X}}}$ 
live in $\C^p$. There is a unique 
\[
x_{j,k}\in\Null\left(\widehat{\mathcal{A}}^{(k)}-
\widehat{\bm{\lambda}}_j^{(k)}I\right)^{\ell_i}
\]
such that 
\[
\widehat{\vec{\mathcal{X}}}^{(k)}=\sum_{j=1}^{\ell}x_{j,k}.
\]
The sequences $\{x_{1,1},\ldots,x_{\ell,1}\},\; \ldots, \; \{x_{1,n},\ldots,x_{\ell,n}\}$
determine the sequence of lateral slices 
$\{\vec{\mathcal{X}}_1,\ldots,\vec{\mathcal{X}}_\ell\}$, such that
\[
\vec{\mathcal{X}}=\sum_{j=1}^{\ell}\vec{\mathcal{X}}_j,
\]
with $\vec{\mathcal{X}}_j\in\tNull\left(\mathcal{A}-\bm{\lambda}_i * 
\mathcal{I}_p\right)^{\ell_i}$.
\end{proof}

\section{Methods for eigenpair computation of third-order tensors} \label{sec 5}
This section describes several methods for computing one or several eigenpairs of 
third-order tensors.

\subsection{Single lateral slice iteration}
The power method is one of the oldest and simplest techniques for computing the
eigenvalue of largest magnitude and an associated eigenvector of a square matrix; see,
e.g., \cite{golub2000eigenvalue}. It is also known as von Mises iteration \cite{vonmises}. 
Presently, this method finds applications to page-rank computations; see
\cite{boubekraoui,BP}. We describe a t-power method for computing the eigenslice of 
largest magnitude of a third-order tensor, and discuss some of its properties. These
properties are well known in matrix computations. Kilmer et al. \cite{kilmer2013third} 
present an algorithm that is slightly different from the t-power method described below, 
however, without defining eigentubes and eigenslices.

\subsection{The t-power method}
This method determines a sequence of lateral slices $\mathcal{A}^k*\vec{\mathcal{V}}_0$, 
$k=1,2,\ldots~$, suitably scaled, where $\vec{\mathcal{V}}_0\in\K^p_n$. Under suitable 
conditions on $\mathcal{A}$ and $\vec{\mathcal{V}}_0$ (see below), the sequence generated
converges to an eigenslice associated with the eigentube of largest norm. Here and below 
``norm'' refers to the Frobenius norm. The main advantage of the t-power method, when 
compared with other methods for computing eigenpairs, is the fairly small amount of 
computer memory required. Algorithm \ref{alg 3} describes the computations with the 
t-power method.
	
\begin{algorithm}[H]
\caption{The t-power method}\label{alg 3}
\textbf{Input:} $\mathcal{A}\in \K^{p\times p}_n$, $\vec{\mathcal{V}}_0\in \K^{p}_n$ 
nonzero.\\
\textbf{Output:} An eigenpair $\left\{\bm{\lambda},\vec{\mathcal{U}}\right\}$ of 
$\mathcal{A}$, where $\bm{\lambda}$ is the eigentube of largest norm.
\begin{algorithmic}[1]
\FOR{$k=1,2,\ldots~$ until convergence}
\STATE $\bm{\alpha}_k=\text{t-}\max(\mathcal{A}*\vec{\mathcal{V}}_{k-1})$
\STATE $\vec{\mathcal{V}}_k=(\mathcal{A}*\vec{\mathcal{V}}_{k-1})/\bm{\alpha}_k$
\ENDFOR
\STATE $\vec{\mathcal{U}}=\vec{\mathcal{V}}_k$, $\bm{\lambda}=\bm{\alpha}_k$
\end{algorithmic}
\end{algorithm}

The function $\text{t-}\max~$ of a lateral slice in line $2$ of Algorithm \ref{alg 3} 
returns the tube of largest norm of all tubes in the lateral slice. The stopping criterion
will be discussed in Section \ref{sec 6}. The following result generalizes the classical 
convergence result for the matrix power method to the t-power method. The proof is 
analogous to the matrix case and therefore is omitted.

\begin{theorem}\label{theo t-pm}
Let $\mathcal{A}\in\K^{p\times p}_n\,$ have only one eigentube, $\bm{\lambda}_1$, of 
largest norm. The initial lateral slice $\vec{\mathcal{V}}_0$ for the t-power method 
either has no tube in the invariant subspace associated with $\bm{\lambda}_1$, or the 
sequence $\vec{\mathcal{V}}_k$, $k=1,2,\ldots~$, generated by Algorithm \ref{alg 3} 
converges to an eigenslice associated with $\bm{\lambda}_1$, and the sequence 
$\bm{\alpha}_k$, $k=1,2,\ldots~$, converges to $\bm{\lambda}_1$.
\end{theorem}

\begin{theorem}\label{coro 4}
Let $\mathcal{A}\in \K^{p\times p}_n$ have $p$ distinct eigentubes 
$\bm{\lambda}_1,\bm{\lambda}_2,\ldots,\bm{\lambda}_p$ ordered increasingly, i.e., 
$\left\Vert \bm{\lambda}_1\right\Vert_F> \left\Vert \bm{\lambda}_2\right\Vert_F\geq\ldots
\geq \left\Vert \bm{\lambda}_p\right\Vert_F$. Let the lateral slice $\vec{\mathcal{V}}_0$ 
have a component in the invariant subspace associated to $\bm{\lambda}_1$. Then the 
sequence $\vec{\mathcal{V}}_k$, $k=1,2,\ldots~$, generated by Algorithm \ref{alg 3} 
converges to an eigenslice associated with the eigentube $\bm{\lambda}_1$.
\end{theorem}

\begin{proof}
By Proposition \ref{propsum}, we have
\[
\vec{\mathcal{V}}_0= \sum_{i=1}^{p}\bm{c}_i *\vec{\mathcal{X}}_i,
\]
where $\vec{\mathcal{X}}_i\in\tNull\left(\mathcal{A}-\bm{\lambda}_i*\mathcal{I}_p\right)$.
Since $\vec{\mathcal{V}}_0$ has a component in the invariant subspace associated with
$\bm{\lambda}_1$, it follows that $\bm{c}_1\neq 0$.

We are interested in the space spanned by $\vec{\mathcal{V}}_k$. This space is independent
of the scaling factors $\bm{\alpha}_k\ne\bm{0}$ in Algorithm \ref{alg 3}. Therefore we may
assume in this proof may ignore the $\bm{\alpha}_k$ for all $k$. We then have for 
$k=1,2,\ldots~$ that
\[ 
\vec{\mathcal{V}}_k = \mathcal{A}^k *\vec{\mathcal{V}}_0 = 
\sum_{i=1}^{p} \bm{c}_i * \bm{\lambda}_i^k * \vec{\mathcal{X}}_i= 
\bm{c}_1 * \bm{\lambda}_1^k * \vec{\mathcal{X}}_1 + \sum_{i=2}^{p} 
\left(\dfrac{\bm{\lambda}_i}{\bm{\lambda}_1}\right)^k*\bm{c}_i *\vec{\mathcal{X}}_i.
\]
Therefore
\[
\begin{split}
\left\Vert \vec{\mathcal{V}}_k - \bm{c}_1*\bm{\lambda}_1^k*
\vec{\mathcal{X}}_1 \right\Vert_F= \left\Vert \sum_{i=2}^{p} 
\left(\dfrac{\bm{\lambda}_i}{\bm{\lambda}_1}\right)^k*
\bm{c}_i*\vec{\mathcal{X}}_i \right\Vert_F 
&\leq \sum_{i=2}^{p} \left\Vert  \left(\dfrac{\bm{\lambda}_i}{\bm{\lambda}_1}\right)^k*
\bm{c}_i*\vec{\mathcal{X}}_i \right\Vert_F\\
&\leq \sum_{i=2}^{p} \left\Vert \dfrac{\bm{\lambda}_i}{\bm{\lambda}_1}\right\Vert_F^k 
\left\Vert \bm{c}_i * \vec{\mathcal{V}}_i \right\Vert_F.
\end{split}
\]
Since $\left\Vert\bm{\lambda}_i\right\Vert_F<\left\Vert \bm{\lambda}_1\right\Vert_F$ for 
all $i>1$, the sum in the right-hand side converges to zero as $k$ increases. It follows
that 
\[
\lim_{k\to+\infty}\vec{\mathcal{V}}_k =\lim_{k\to+\infty} \bm{c}_1 *\bm{\lambda}_1^k * 
\vec{\mathcal{X}}_1.
\]
This shows the theorem.
\end{proof}
	
We can see from the proof of Theorem \ref{coro 4}, that the rate of convergence is larger,
the smaller the quotients 
$\left\Vert\bm{\lambda}_i\right\Vert_F/\left\Vert\bm{\lambda}_1\right\Vert_F$ are for 
$i=2,\ldots,p$. 
	
The rate of convergence of the t-power method may be improved by applying the method to 
the shifted tensor $\mathcal{A}+\bm{\sigma}*\mathcal{I}_p$ for a suitable shift 
$\bm{\sigma}$, instead of applying the method to $\mathcal{A}$. This follows from the
proof of Theorem \ref{coro 4}. We omit the details.

\subsection{The inverse t-power method} 
In its simplest form this method can be applied to compute an eigenslice associated with 
the eigentube of smallest norm. Then the method is obtained by applying the t-power method
to the tensor $\mathcal{A}^{-1}$. This yields the sequence
\[
\vec{\mathcal{V}}_k=(\mathcal{A}^{-1}* \vec{\mathcal{V}}_{k-1})/\bm{\alpha}_k,\quad 
k=1,2,\ldots~,
\]
of lateral slices, where $\vec{\mathcal V}_0$ is an initial user-supplied slice. The
sequence converges to the lateral slice associated with the eigentube of largest norm
of $\mathcal{A}^{-1}$, if this eigentube is unique and if $\vec{\mathcal V}_0$ contains a
component of this slice. Since the eigentubes of $\mathcal{A}$ are the inverse of the 
eigentubes of $\mathcal{A}^{-1}$, we obtain converge to the lateral slice associated with 
the eigentube of smallest norm of $\mathcal{A}$. Algorithm \ref{alg 4} describes a shifted
version that can be used to determine an eigenslice associated with the eigentube that is
closest to $\bm{\sigma}$, if this eigentube is unique and the initial slice 
$\vec{\mathcal V}_0$ contains a component of this eigenslice.

\begin{algorithm}[H]
\caption{The shifted inverse t-power method}\label{alg 4}
\textbf{Input:} $\mathcal{A}\in\K^{p\times p}_n$, $\vec{\mathcal{V}}_0\in\K^{p}_n$, and 
$\bm{\sigma}\in \mathbb{K}_n$.\\
\textbf{Output:} The eigenpair $\{\bm{\lambda},\vec{\mathcal{U}}\}$ of $\mathcal{A}$, 
where $\bm{\lambda}$ is the eigentube closest to $\bm{\sigma}$.\vskip1pt 
\begin{algorithmic}[1]
\FOR{$k=1,2,\ldots~$ until convergence}
\STATE Evaluate $\vec{\mathcal{W}}=
\left(\mathcal{A}-\bm{\sigma}*\mathcal{I}_p\right)^{-1}*\vec{\mathcal{V}}_{k-1}$
\STATE $\bm{\alpha}_{k}=\text{t-}\max(\vec{\mathcal{W}})$
\STATE $\vec{\mathcal{V}}_k=\vec{\mathcal{W}}/\bm{\alpha}_k$
\ENDFOR
\STATE $\vec{\mathcal{U}}=\vec{\mathcal{V}}_k$, 
$\bm{\lambda}=\dfrac{\bm{e}}{\bm{\alpha}_k+\bm{\sigma}}$
\end{algorithmic}
\end{algorithm}

The evaluation of $\vec{\mathcal{W}}$ in line 2 of Algorithm \ref{alg 4} can be carried 
out, e.g., by using a t-QR, t-LU, or t-Choleski factorization of 
$\mathcal{A}-\bm{\sigma}*\mathcal{I}_p$; the t-Choleski factorization is described in 
\cite{RU3}.

\subsection{Deflation}\label{subs 5.2}
Deflation is a well known technique in matrix computation. It makes it possible to use the
power method to compute more eigenpairs of a matrix than the pair associated with the 
eigenvalue of largest magnitude. This section describes how deflation can be used together
with the t-power method. 

Let the tensor $\mathcal{A}\in\K^{p\times p}_{n}$ have a unique eigentube 
$\bm{\lambda}_1$ of largest norm, and let $\vec{\mathcal{U}}_1$ be an associated 
eigenslice. Assume that the eigenpair $\{\bm{\lambda}_1,\vec{\mathcal{U}}_1\}$ has been 
computed, e.g., by the t-tensor power method. The idea behind deflation is to apply 
Algorithm \ref{alg 3}~to the tensor
\begin{equation}\label{deflation}
\mathcal{A}_1=\mathcal{A}-\bm{\sigma}*\vec{\mathcal{U}}_1*\vec{\mathcal{V}}^H,
\end{equation}
where $\vec{\mathcal{V}}\in\K^{p}_n$ is an arbitrary lateral slice such that 
$\vec{\mathcal{V}}^H*\vec{\mathcal{U}}_1=\bm{e}$, and $\bm{\sigma}\in\mathbb{K}_n$ is an
appropriate shift. The tensor $\mathcal{A}_1$ has the same eigentubes as $\mathcal{A}$
except for $\bm{\lambda}_1$, which is transformed to $\bm{\lambda}_1-\bm{\sigma}$, as is 
shown in the next theorem. 

\begin{theorem}
Let $\mathcal{A}\in\K^{p\times p}_n$ have the eigenpairs
$\{\bm{\lambda}_i,\vec{\mathcal{U}}_i\}_{i=1}^\ell$, and let the tube 
$\bm{\sigma}\in\K_n$ be such that $\bm{\lambda}_i\ne\bm{\lambda}_1 -\bm{\sigma}$, 
$i=1,2,\ldots,\ell$. Then the tensor $\mathcal{A}_1$, defined by \eqref{deflation}, has 
the eigenpairs 
\[
\{(\bm{\lambda}_1-\bm{\sigma},\vec{\mathcal{U}}_1), (\bm{\lambda}_i, 
\vec{\mathcal{U}}_i+\bm{\gamma}_i*\vec{\mathcal{U}}_1)_{i=2}^\ell\},
\]
with 
\[
\bm{\gamma}_i=\dfrac{\bm{\sigma}*\vec{\mathcal{V}}^H*\vec{\mathcal{U}}_i}
{\bm{\sigma}-(\bm{\lambda}_1-\bm{\lambda}_i)},\quad i=2,3,\ldots,\ell.
\]
\end{theorem} 

\begin{proof}
First note that $\bm{\lambda}_i\ne\bm{\lambda}_1 -\bm{\sigma}$ means that 
\[
\widehat{\bm{\lambda}}_i^{(j)}\neq\widehat{\bm{\lambda}}_1^{(j)}-
\widehat{\bm{\sigma}}^{(j)},\quad 1\leq j\leq n. 
\]
Let $\bm{\sigma}$ be such that $\bm{\lambda}_i\ne\bm{\lambda}_1-\bm{\sigma}$. We will show
that $\{\bm{\lambda}_1 -\bm{\sigma}, \vec{\mathcal{U}}_1\}$ is an eigenpair of 
$\mathcal{A}_1$ and that the deflation procedure preserves the eigentubes 
$\bm{\lambda}_2,\bm{\lambda}_3,\ldots,\bm{\lambda}_\ell$. We have 
\[
\begin{split}
\mathcal{A}_1*\vec{\mathcal{U}}_1=\left(\mathcal{A}-\bm{\sigma}*\vec{\mathcal{U}}_1*
\vec{\mathcal{V}}^H\right)*\vec{\mathcal{U}}_1&= 
\mathcal{A}*\vec{\mathcal{U}}_1 - \bm{\sigma}*\vec{\mathcal{U}}_1*\vec{\mathcal{V}}^H 
*\vec{\mathcal{U}}_1\\
 &= \bm{\lambda}_1*\vec{\mathcal{U}}_1 - \bm{\sigma}*\vec{\mathcal{U}}_1\\
 &=\left(\bm{\lambda}_1 - \bm{\sigma}\right)*\vec{\mathcal{U}}_1.
\end{split}
\]
On the other hand, let $\vec{\mathcal{W}}_j$  be the left eigenslice associated with 
$\bm{\lambda}_j$ for $j=2,3,\ldots,\ell$. Then we have 
\[
\mathcal{A}^H*\vec{\mathcal{W}}_j=\bm{\lambda_j}*\vec{\mathcal{W}}_j,
\]
and since $\{\vec{\mathcal{W}}_j\}_{j=2}^\ell$ are f-orthogonal to $\vec{\mathcal{U}}_1$, we get 
\[
\begin{split}
\mathcal{A}_1^H * \vec{\mathcal{W}}_j=\left(\mathcal{A}-\bm{\sigma}*\vec{\mathcal{U}}_1*
\vec{\mathcal{V}}^H\right)^H *\vec{\mathcal{W}}_j 
&=\left(\mathcal{A}^H - \bm{\sigma}^H * \vec{\mathcal{V}}*\vec{\mathcal{U}}_1^H\right)*
\vec{\mathcal{W}}_j\\
&= \bm{\lambda}_j * \vec{\mathcal{W}}_j,\quad j=2,3,\ldots,\ell.
\end{split}
\]
Thus, deflation preserves the left eigenslices 
$\vec{\mathcal{W}}_2,\vec{\mathcal{W}}_3,\ldots,\vec{\mathcal{W}}_\ell$ and the 
eigentubes $\bm{\lambda}_2,\bm{\lambda}_3,\ldots,\bm{\lambda}_\ell$.

Since $\bm{\lambda}_2,\bm{\lambda}_3,\ldots,\bm{\lambda}_\ell$ are eigentubes of
$\mathcal{A}_1$, we will now construct eigenslices associated with these eigentubes. Let 
for $i=2,3,\ldots,\ell$,
\[
\check{\vec{\mathcal{U}}}_i=\vec{\mathcal{U}}_i - \bm{\gamma}_i*\vec{\mathcal{U}}_1.
\]
Since  
$\mathcal{A}_1 * \check{\vec{\mathcal{U}}}_i=\bm{\lambda}_i * \check{\vec{\mathcal{U}}}_i$, it follows that $\bm{\gamma}_i$  is given by 
\[
\bm{\gamma}_i=\dfrac{\bm{\sigma}*\vec{\mathcal{V}}^H * \vec{\mathcal{U}}_i}
{\bm{\sigma}-\left(\bm{\lambda}_1-\bm{\lambda}_i\right)},\quad i=2,3,\ldots,\ell, 
\]
and $\bm{\gamma}_1=0$ such that $\bm{\sigma}\ne\bm{\lambda}_1 - \bm{\lambda}_i$.  
\end{proof}
	
\smallskip
Deflation techniques for matrices have been discussed in several works, see, e.g., Saad
\cite[Chapter 4.2]{Saad_Yousef}, and the choice of the analogue of the slice 
$\vec{\mathcal{V}}$ has received some attention. Analyses for the matrix case suggest that
we let $\vec{\mathcal{V}}=\vec{\mathcal{W}}_1$, where $\vec{\mathcal{W}}_1$ is the left 
eigenslice of $\mathcal{A}$, or $\vec{\mathcal{V}}=\vec{\mathcal{U}}_1$. The latter choice 
preserves the t-Schur lateral slices.
\smallskip

\begin{proposition}
Let $\vec{\mathcal{U}}_1$ be an eigenslice of $\mathcal{A}$ of norm $1$ associated with 
$\bm{\lambda}_1$ and let 
$\mathcal{A}_1=\mathcal{A}-\bm{\sigma}*\vec{\mathcal{U}}_1*\vec{\mathcal{U}}_1^H$. Then 
the eigentubes of $\mathcal{A}_1$ are 
$\widetilde{\bm{\lambda}}_1=\bm{\lambda}_1-\bm{\sigma}$ and 
$\widetilde{\bm{\lambda}}_j=\bm{\lambda}_j$, $j=2,3,\ldots,\ell$. Moreover, the t-Schur 
lateral slices associated with $\widetilde{\bm{\lambda}}_j$, $j=1,2,\ldots,\ell$ are 
identical with those of $\mathcal{A}$.
\end{proposition}
\smallskip

\begin{proof}
Let $\mathcal{A}* \mathcal{U}= \mathcal{U}*\mathcal{R}$ be the t-Schur factorization of 
the tensor $\mathcal{A}$, where $\mathcal{U}\in\K^{p\times p}_n$ is an f-orthogonal tensor
and $\mathcal{R}\in \K^{p\times p}_n$ is an upper f-triangular tensor. Thus,
\[
\mathcal{A}_1 * \mathcal{U}=
\left(\mathcal{A}-\vec{\mathcal{U}}_1* \vec{\mathcal{U}}_1^H\right)*\mathcal{U}= 
\mathcal{U}*\mathcal{R}-\bm{\sigma}*\vec{\mathcal{U}}_1 * \vec{\mathcal{E}}_1^H = 
\mathcal{U}*\left(\mathcal{R}-\bm{\sigma}*\vec{\mathcal{E}}_1*\vec{\mathcal{E}}_1^H
\right), 
\]
which is the desired result since 
$\mathcal{R}-\bm{\sigma}*\vec{\mathcal{E}}_1*\vec{\mathcal{E}}_1^H$ is an upper 
f-triangular tensor.
\end{proof}
	
Deflation can be applied multiple times using several vectors. For the matrix case this is
described in, e.g., \cite[Chapter 4.2.3]{Saad_Yousef}. The techniques can be adapted to 
the tensor case. We omit the details.

\subsection{The tensor subspace iteration method}
Subspace iteration is a powerful technique for computing a few eigenpairs with the largest
eigenvalues of a large matrix; see, e.g., \cite{lehoucq,Saad_Yousef}. This section 
describes the application of subspace iteration to the computation of a few eigenpairs 
with the largest eigentubes of a large tensor $\mathcal{A}\in\K^{p\times p}_n$. The method
is described by Algorithm \ref{alg 5}. 
	
\begin{algorithm}[H]
\caption{t-subspace iteration}\label{alg 5}
\textbf{Input:} $\mathcal{A}\in\mathbb{K}^{p\times p}_n$,  
$\mathcal{X}_0\in\K^{p\times s}_n$, integer $q>0$.\\
\textbf{Output:} Upper f-triangular tensor $\mathcal{R}\in\K^{s\times s}_n$ and an 
f-unitary tensor $\mathcal{U}\in\K^{p\times s}_n$ such that 
$\mathcal{R}=\mathcal{U}^H * \mathcal{A}*\mathcal{U}$ contains the first $s$ eigentubes on
its diagonal.  
\begin{algorithmic}[1]
\FOR {$k=1,2,\ldots~$ until convergence}
	\STATE $\mathcal{X}_k= \mathcal{A}^q*\mathcal{X}_{k-1}$
	\STATE $\mathcal{X}_k=\mathcal{Q}_k*\mathcal{R}_k$
	\STATE $\mathcal{X}_k=\mathcal{Q}_k$
\ENDFOR
\STATE $\mathcal{U}=\mathcal{X}_k$, $\mathcal{R}=\mathcal{U}^H *\mathcal{A}*\mathcal{U}$
\end{algorithmic}
\end{algorithm}

The positive integer $q$ in line 2 of Algorithm \ref{alg 5} determines the number
of products with the tensor $\mathcal{A}$ in each iteration. A value $q>1$ reduces
the number of t-QR factorizations required to achieve a given power of $\mathcal{A}$
and, therefore, may speed up the computations. The t-QR factorization of the tensor 
${\mathcal X}_k\in\K^{p\times s}_n$ is computed in line 3 of Algorithm \ref{alg 5}.
The frontal slices of the tensor ${\mathcal X}_k\in\K^{p\times s}_n$ converge to the 
desired eigenslices under suitable conditions described in the following theorem.
\smallskip

\begin{theorem}\label{theo ss}
Let the eigentubes $\bm{\lambda}_1,\bm{\lambda}_2,\ldots,\bm{\lambda}_n$ of the tensor
$\mathcal{A}\in\K^{p\times p}_n$ satisfy 
\[
\|\bm{\lambda}_1\|>\|\bm{\lambda}_2\|>\ldots>\|\bm{\lambda}_m\|>\|\bm{\lambda}_{m+1}\|
\geq\|\bm{\lambda}_{m+2}\|\geq\ldots\geq\|\bm{\lambda}_n\|.
\]
Introduce 
$\mathcal{Q}=[\vec{\mathcal{Q}}_1,\ldots,\vec{\mathcal{Q}}_m]\in\K^{p\times m}_n$,
where $\vec{\mathcal{Q}}_i$ is the t-Schur lateral slice associated with 
$\bm{\lambda}_i$, $1\leq i\leq m$, and let $\mathcal{P}_i$ denote the spectral 
projector-tensor associated with $\bm{\lambda}_1,\ldots, \bm{\lambda}_m$, and let the 
initial iterate be 
$\mathcal{X}_0=[\vec{\mathcal{X}}_1,\ldots,\vec{\mathcal{X}}_m]\in\K^{p\times m}_n$. 
Assume that 
\[
{\rm rank}\left( \mathcal{P}_i*\left[\vec{\mathcal{X}}_1,\ldots,\vec{\mathcal{X}}_i\right]
\right)=i,\quad i=1,2,\ldots,m.
\]
Then the the $\tspan$ of the $i^{th}$ lateral slice $\vec{\mathcal{X}}_{k,i}$ of 
$\mathcal{X}_k$ generated by Algorithm \ref{alg 5} converges to 
$\tspan\{\vec{\mathcal{Q}_i}\}$.
\end{theorem}

\begin{proof}
The proof is analogous to the matrix case; see Saad \cite[Theorem 5.1]{Saad_Yousef} for
the latter.
\end{proof}

We remark that subspace iteration for matrices can be enhanced in a variety of ways, e.g.,
by including projections and deflation; see Saad \cite[Theorem 5.1]{Saad_Yousef}. 
Algorithm \ref{alg 5} can be enhanced similarly.
\smallskip

\subsection{The t-QR algorithm for tensors} 
The QR algorithm is the default method for computing all eigenpairs of a matrix of small 
to moderate size. A fairly recent discussion of the QR algorithm for matrices with 
references is provided in, e.g., \cite{Golub}. This section describes an extension to the 
computation of eigenpairs of third-order tensors. This extension is referred to as the 
t-QR algorithm.

The default QR algorithm used for matrix computation is fairly complicated: it uses 
implicit shifts, and double shifts for real nonsymmetric matrices. In the interest of
brevity, we only discuss a basic version of the t-QR algorithm. Details of a more
efficient implementation that uses implicit and double shifts will be described
elsewhere.

The unshifted t-QR algorithm described by Algorithm \ref{alg 6} is not applied in 
practice due to its slow convergence. Our interest in the algorithm stems from that it is
quite straightforward to show properties of the tensors it computes. Analogous properties
can be shown for the tensors determined by the shifted algorithm discussed below.

\begin{algorithm}[H]
\caption{The unshifted {\tt t-QR} algorithm}\label{alg 6}
\textbf{Input:} $\mathcal{A}_0=\mathcal{A}\in\K^{p\times p}_n$.\\
\textbf{Output:} The eigentubes of $\mathcal{A}$. 
\begin{algorithmic}[1]
\FOR {$k=1$ until the convergence}
\STATE Compute $\mathcal{A}_k=\mathcal{Q}_k * \mathcal{R}_k$
\STATE Compute $\mathcal{A}_{k+1}=\mathcal{R}_k * \mathcal{Q}_k$
\ENDFOR	
\end{algorithmic}
\end{algorithm}
\smallskip

\begin{proposition}
Let $\mathcal{A}\in\K^{p\times p}_n$. Then the tensors 
$\mathcal{A}_1,\mathcal{A}_2,\ldots~$ generated by Algorithm \ref{alg 6} satisfy 
\begin{equation}\label{eq 2.19}
\mathcal{A}_{k+1} =\left(\mathcal{Q}_1 * \mathcal{Q}_2 * \ldots * \mathcal{Q}_k \right)^H 
* \mathcal{A} * \left(\mathcal{Q}_1 * \mathcal{Q}_2 * \ldots * \mathcal{Q}_k \right).
\end{equation}
Moreover,
\begin{equation}\label{eq 2.20}
\mathcal{A}^k = \left(\mathcal{Q}_1 * \mathcal{Q}_2 * \ldots * \mathcal{Q}_k \right)* 
\left(\mathcal{R}_{k} * \mathcal{R}_{k-1} * \ldots * \mathcal{R}_1 \right), 
\end{equation}
where the sequences $\{\mathcal{Q}_i\}_{i=1}^k$ and $\{\mathcal{R}_i\}_{i=1}^k$ are 
defined by the {\tt t-QR} factorization in line 2 of Algorithm \ref{alg 6}. 
\end{proposition} 
\smallskip

\begin{proof}
Let $k\geq 1$. Then 
\[
\begin{split}
\mathcal{A}_{k+1}=\mathcal{R}_{k}*\mathcal{Q}_{k}&=\mathcal{Q}_{k}^H * 
\mathcal{Q}_{k}*\mathcal{R}_{k}*\mathcal{Q}_{k}=\mathcal{Q}_{k}^H * \mathcal{A}_{k}* 
\mathcal{Q}_{k}\\
 &= \mathcal{Q}_{k}^H * \mathcal{R}_{k-1}* \mathcal{Q}_{k-1}* \mathcal{Q}_{k}\\
 &= \mathcal{Q}_{k}^H * \mathcal{Q}_{k-1}^H * \mathcal{A}_{k-2}* \mathcal{Q}_{k-1}* 
 \mathcal{Q}_{k}\\
 &= \left(\mathcal{Q}_1 * \mathcal{Q}_2 * \ldots * \mathcal{Q}_k \right)^H \mathcal{A} 
 \left(\mathcal{Q}_1 * \mathcal{Q}_2 * \ldots * \mathcal{Q}_k \right), 
\end{split}
\]
which gives \eqref{eq 2.19}. On the other hand, 
\[
\begin{split}
\mathcal{A}^k&= \mathcal{Q}_1 * \mathcal{R}_1 * \ldots * \mathcal{Q}_1 * \mathcal{R}_1 \\
 &= \mathcal{Q}_1 * \mathcal{Q}_2 * \mathcal{R}_2 * \ldots * \mathcal{Q}_2 * \mathcal{R}_2
 * \mathcal{R}_1\\
 &= \left(\mathcal{Q}_1 * \ldots * \mathcal{Q}_k\right)* \left(\mathcal{R}_k * \ldots * 
 \mathcal{R}_1\right), 
\end{split}
\]
which shows \eqref{eq 2.20}.
\end{proof}
	
Eq. \eqref{eq 2.19} shows that the tensor $\mathcal{A}_{k+1}\ $ is similar to 
$\mathcal{A}$, while relation \eqref{eq 2.20} provides a {\tt t-QR} factorization of the
tensor $\mathcal{A}^k$. Letting
\begin{equation}\label{rrqq}
\breve{\mathcal{Q}}_k=\mathcal{Q}_1*\mathcal{Q}_2*\ldots * \mathcal{Q}_k \;\text{and }\;  
\breve{\mathcal{R}}_k=\mathcal{R}_k*\mathcal{R}_{k-1}*\ldots * \mathcal{R}_1,
\end{equation}
we obtain 
\[
\mathcal{A}_{k+1}=\breve{\mathcal{Q}}_k^H *\mathcal{A}_k * \breve{\mathcal{Q}}_k.
\]
Thus, the tensor $\mathcal{A}_{k+1}\ $ is f-unitarily similar to $\mathcal{A}_k$.

The following result is concerned with the convergence of Algorithm \ref{alg 6}. 
Convergence is considered under special conditions, but can be shown under less
restrictive conditions as well.
\smallskip

\begin{theorem}
Let $\mathcal{A}\in\K^{n\times n}_p$ be such that
\[
\mathcal{X}^{-1}*\mathcal{A}*\mathcal{X}=
\bm{\Sigma}=\tdiag\left[ \bm{\lambda}_1, \bm{\lambda}_2, \ldots, \bm{\lambda}_n \right],
\]
with 
\begin{equation}\label{ordering}
\left\Vert \bm{\lambda}_1\right\Vert_F > \left\Vert \bm{\lambda}_2\right\Vert_F >\ldots > 
\left\Vert \bm{\lambda}_n\right\Vert_F > 0,
\end{equation}
and assume that all entries of every tube $\bm{\lambda}_k$ are nonvanishing. Let 
$\mathcal{X}^{-1}=\mathcal{L}*\mathcal{U}$ denote the t-{\tt LU} factorization of 
$\mathcal{X}^{-1}$, where $\mathcal{L}$ is a lower f-triangular tensor with the elements 
$\textbf{e}$ on its diagonal; we assume for notational simplicity that pivoting is not 
required. Let $\mathcal{X}=\mathcal{Q}*\mathcal{R}$ be a t-{\tt QR} factorization. Then 
there is an invertible f-diagonal tensor $\mathcal{D}_k$, such that 
$\breve{\mathcal{Q}}_k * \mathcal{D}_k^{-1} $ converges to $\mathcal{Q}$ as $k$ increases, 
where $\breve{\mathcal{Q}}_k$ is given by \eqref{rrqq}.
\end{theorem}
\smallskip

\begin{proof}
We have 
\[
\begin{split}
\mathcal{A}^k = \mathcal{X}* \bm{\Sigma}^k * \mathcal{X}^{-1}
&= \mathcal{Q}*\mathcal{R}*\bm{\Sigma}^k * \mathcal{L}*\mathcal{U}\\
&= \mathcal{Q}*\mathcal{R}*\bm{\Sigma}^k * \mathcal{L}*\bm{\Sigma}^{-k}*\bm{\Sigma}^{k}*
   \mathcal{U},
\end{split}
\]
and we note that the elements of $\bm{\Sigma}^{k}*\mathcal{L}*\bm{\Sigma}^{-k}$ are 
given by
\[
\bm{\Sigma}^{k}*\mathcal{L}*\bm{\Sigma}^{-k}(i,j,:)=\begin{cases}
\mathcal{L}(i,j,:)*\left(\bm{\lambda}_i/\bm{\lambda}_j\right)^{k}\;\; \text{if} \; i>j,\\
	\textbf{e}\;\; \text{if} \; i=j,\\
	\textbf{0} \;\; \text{otherwise}.
\end{cases}
\]
Then by using the ordering \eqref{ordering} of the eigentubes, the tensor  
$\bm{\Sigma}^{k}*\mathcal{L}*\bm{\Sigma}^{-k}$ converges to $\mathcal{I}$ as $k$ 
increases. Thus, $\bm{\Sigma}^{k}*\mathcal{L}*\bm{\Sigma}^{-k}$ can be written as 
$\mathcal{I}+ \mathcal{E}_k$, where $\mathcal{E}_k \to \textbf{0}$ as $k$ increases.

Since $\mathcal{X}$ is invertible, $\mathcal{R}$ also is invertible. Therefore 
\[
\begin{split}
\mathcal{A}^k&=\mathcal{Q}*\mathcal{R}*\left(\mathcal{I}+\mathcal{E}_k\right)*
\bm{\Sigma}^k*\mathcal{U}\nonumber\\
 &= \mathcal{Q}*\left( \mathcal{I}+\mathcal{R}*\mathcal{E}_k * \mathcal{R}^{-1}\right)*
 \mathcal{R}*\bm{\Sigma}^k *\mathcal{U}.
\end{split}
\]
Introduce the t-QR factorization $\mathcal{Q}_k' *\mathcal{R}'_k=
\mathcal{I}+\mathcal{R}*\bm{\Sigma}_k *\mathcal{R}^{-1}$. Then
\[
\mathcal{A}^k= \mathcal{Q}*\mathcal{Q}_k' * \mathcal{R}_k' * 
\mathcal{R}*\bm{\Sigma}^k*\mathcal{U}.
\]
Let $\bm{\delta}_1, \bm{\delta}_2, \ldots, \bm{\delta}_n$ denote the diagonal elements of 
the tensor $\mathcal{R}_k' * \mathcal{R}* \bm{\Sigma}^k * \mathcal{U}$. Since the elements 
on the diagonals of the tensors $\mathcal{R}_k'$, $\mathcal{R}$, $\bm{\Sigma}^k$, and 
$\mathcal{U}$ are nonvanishing, it follows that the tensor
\[
\mathcal{D}_k = \tdiag\left( \bm{\delta}_1, \bm{\delta}_2, \ldots, \bm{\delta}_n 
\right)\in \K^{n\times n}_p
\]
is invertible. We therefore can write
\[
\mathcal{A}^k = \left(\mathcal{Q}*\mathcal{Q}_k'*\mathcal{D}_k\right)*
\left(\mathcal{D}_k^{-1}*\mathcal{R}_k'*\mathcal{R}*\bm{\Sigma}^{k}*\mathcal{U}\right).
\]
Then, by using the properties of the {\tt t-QR} factorization, we conclude that 
\[
\breve{\mathcal{Q}}_k=\mathcal{Q}*\mathcal{Q}_k' * \mathcal{D}_k.
\]
Moreover, since $\mathcal{Q}_k'$ converges to $\mathcal{I}$ as $k$ increases, the tensor 
sequence
\[
\breve{\mathcal{Q}}_k * \mathcal{D}_k^{-1} = \mathcal{Q}* \mathcal{Q}_k',\quad 
k=1,2,\ldots~,
\]
converges to $\mathcal{Q}$.

Since the lateral slices of the tensor $\mathcal{Q}$ are the t-Schur lateral slices of 
$\mathcal{A}$, it follows that the scaled t-Schur lateral slices of 
$\mathcal{A}_k$ ($\breve{\mathcal{Q}}_k$) with elements in $\K_p$ converge to the lateral 
slices of $\mathcal{A}$ ($\mathcal{Q}$) as $k$ increases. Consequently,
$\mathcal{A}_k=\breve{\mathcal{Q}}_k^H * \mathcal{A}* \breve{\mathcal{Q}}_k$ converges to
the upper f-triangular tensor related to the t-Schur lateral slices. 
\end{proof}
	
\smallskip
The rate of convergence of Algorithm \ref{alg 6} can be accelerated by introducing 
shifts, i.e., by replacing lines 2 and 3 of the algorithm by 
$\mathcal{A}_k-\bm{\sigma}_k*\mathcal{I}_p=\mathcal{Q}_k * \mathcal{R}_k$ and
$\mathcal{A}_{k+1}=\mathcal{R}_k * \mathcal{Q}_k+\bm{\sigma}_k*\mathcal{I}_p$,
respectively, for suitable tubal shifts $\bm{\sigma}_k$. Moreover, the computational
cost of each iteration can be reduced by first transforming $\mathcal{A}_0$ to 
f-upper-Hessenberg form. These modifications are described by Algorithm \ref{alg 7}.
	
\begin{algorithm}[H]
\caption{A shifted {\tt t-QR} algorithm with f-Hessenberg reduction}\label{alg 7}
\textbf{Input:} $\mathcal{A}_0=\mathcal{A}\in \R^{p\times p}_n$, a shift tube 
$\bm{\sigma}\in\K_n$, $r=p$, $\epsilon>0$.\\
\textbf{Output:} The eigentubes of $\mathcal{A}$.
\begin{algorithmic}[1]
\STATE $[\mathcal{P},\mathcal{H}_0]={\tt t-hess}(\mathcal{A}_0)$, 
with $\mathcal{P}*\mathcal{H}_0 *\mathcal{P}^H=\mathcal{A}_0$
\FOR {$r>1$ until convergence}
\STATE $\mathcal{H}_{k}-\bm{\sigma}*\mathcal{I}_p=\mathcal{Q}_k*\mathcal{R}_k$\quad
(Compute QR factorization)
\STATE $\mathcal{A}_{k+1}=\mathcal{R}_k*\mathcal{Q}_*+\bm{\sigma}*\mathcal{I}_p$
		\IF{$\left\Vert \mathcal{A}_{k+1}(r,r-1,:)\right\Vert_F\leq \epsilon$}
		\STATE $r=r-1$
		\ENDIF
		\ENDFOR
	\end{algorithmic}
\end{algorithm}

One way to determine suitable shifts $\bm{\sigma}$ is to let the shift in each iteration 
be $\mathcal{A}_{k}(r,r,:)$. We should mention that sometimes we need to reorder the 
tubes $\mathcal{A}_{k+1}(j,j-1,:)$, $j=2,3,\ldots,r$, determined in each iteration of 
Algorithm \ref{alg 7} before determining the shift. This issue will be discussed in detail
elsewhere.

\section{Numerical experiments}\label{sec 6}
This section illustrates the performance of the algorithms discussed in the previous
section. The purpose of the examples is to demonstrate that the methods work similarly
as the analogous methods for matrices. We used the following tensors in our experiments: 
The first tensor is 
of size $10\times 10\times 3$ and has the frontal slices
\[
\mathcal{A}(:,:,i)=\delta_i   \; {\tt tridiag}(-1,2,-1),\]	
with $\delta_i=10^{i-1}$ for $1\leq i\leq 3$. The second tensor in our experiments is the
stochastic tensor $\mathcal{C}\in\K^{4\times 4}_4$ with frontal slices
{\small
\[
\mathcal{C}(:,:,1)= \begin{pmatrix}
	0.2091 &   0.2834   & 0.2194  &  0.1830\\
	0.3371   & 0.3997 &   0.3219 &   0.3377\\
	0.3265  &  0.0560  &  0.3119  &  0.2961\\
	0.1273 &   0.2608 &   0.1468   & 0.1832
\end{pmatrix}, \;\mathcal{C}(:,:,2)=\begin{pmatrix}
	0.1952   & 0.2695  &  0.2055 &   0.1690\\
	0.3336   & 0.3962 &   0.3184 &   0.3342\\
	0.2954 &   0.0249 &   0.2808 &   0.2650\\
	0.1758 &   0.3094 &   0.1953 &   0.2318
\end{pmatrix},
\]	
\[
\mathcal{C}(:,:,3)=\begin{pmatrix}
	0.3145  &  0.3887 &   0.3248 &   0.2883\\
	0.0603  &  0.1230 &   0.0451  &  0.0609\\
	0.3960  &  0.1255 &   0.3814 &   0.3656\\
	0.2293   & 0.3628 &   0.2487 &   0.2852
\end{pmatrix},\; \mathcal{C}(:,:,4)=\begin{pmatrix}
	0.1686  &  0.2429 &   0.1789 &   0.1425\\
	0.3553  &  0.4180 &   0.3402 &   0.3559\\
	0.3189   &  0.0484 &   0.3043 &   0.2885\\
	0.1571   & 0.2907  &  0.1766  &  0.2131
\end{pmatrix}.
\]
}
We also used the tensor ${\tt complex}(10,10,10)\in\K^{10\times 10}_{10}$, whose entries 
have normally distributed real and imaginary parts with mean zero and variance one, as 
well as the real tensor ${\tt real}(10,10,10)$, whose entries are normally distributed 
with mean zero and variance one, and with real eigentubes. Table \ref{tab 1} lists the 
abbreviations of the algorithms used in the experiments.

\begin{table}[H]
\centering\small\addtolength{\tabcolsep}{-5pt}
\centering
\begin{tabular}{l|c|c}
\hline Method & Abbreviation & Algorithm  \\ \hline 
t-power method & t-PM & Algorithm \ref{alg 3}\\ \hline 
shifted inverse t-power method& t-SIPM & Algorithm \ref{alg 4}\\ \hline 
t-deflation method& t-DM &  \\ \hline 
t-subspace iteration& t-SI & Algorithm \ref{alg 5}\\ \hline 
t-QR algorithm with f-Hessenberg reduction and shifts~~& t-QRHS & Algorithm \ref{alg 7}\\ \hline
\end{tabular}
\caption{Abbreviations for the algorithms used in the experiments.} \label{tab 1}
\end{table}

Let the diagonal entries of the f-diagonal tensor $\mathcal{D}_k\in\K^{k\times k}_n$ be
the exact eigentubes of the tensors considered, and let the diagonal entries of the
f-diagonal tensor $\widetilde{\mathcal{D}}_k\in\K^{k\times k}_n$ contain the corresponding
approximations computed by the methods. We report the error
\begin{equation} \label{eq err}
	\text{Error}=\Vert\widetilde{\mathcal{D}_k}- \mathcal{D}_k\Vert_F.
\end{equation}

All computations were carried out on a laptop computer with an 2.3 GHz Intel Core i5 
processor and 8 GB of memory using MATLAB 2018a.

\subsection{The tensor t-power method and the inverse t-power method}
This subsection shows the performance of the t-power method and the closely related 
inverse t-power method. These methods are implemented by Algorithms \ref{alg 3} and
\ref{alg 4}. The t-power method determines an approximation of the eigenpair
with the eigentube of largest norm, while the inverse t-power method computes an
approximation of the eigenpair with the eigentube closest to a specified tubal shift.

We first consider the t-power method. Generically, Algorithms \ref{alg 3} determines a 
sequence $\{\vec{\mathcal{V}}_k\}_{k\geq 1}$ of approximations of an eigenslice that is 
associated with the eigentube of largest norm, as well of the sequence 
$\{\bm{\alpha}_k\}_{k\geq 1}$ of approximations of the largest eigentube; see Theorem 
\ref{coro 4}. We will use the stopping criterion  
\[
\Vert \vec{\mathcal{V}}_{k+1}-\vec{\mathcal{V}}_{k}\Vert_F \leq 
\text{tol},\; \text{and}\; \; \left\Vert \bm{\alpha}_{k+1}-\bm{\alpha}_{k}\right\Vert_F
\leq \text{tol},
\]
where we set $\text{tol}=10^{-15}$ in our tests and allow at most $\text{Itermax}=3000$ 
iterations. The initial lateral slice $\vec{\mathcal{V}}_0$ is chosen to have normally 
distributed entries; the entries are chosen to be real if the tensor is real and the 
desired eigenslice is real; otherwise $\vec{\mathcal{V}}_0$ is complex with entries that
have normally distributed real and imaginary parts. Let 
$\{\vec{\mathcal{U}},\bm{\lambda}\}$ denote the computed approximation of the desired
eigenpair of a tensor $\mathcal{F}\in\mathbb{K}^{p\times p}_n$. Then we evaluate the 
residual error norm
\[
\text{Res.norm}=\Vert\mathcal{F}*\vec{\mathcal{U}}-\vec{\mathcal{U}}*\bm{\lambda}\Vert_F.
\].
Table \ref{tab 2} reports some computed results. As can be expected, the t-power method 
may require many iterations to satisfy the convergence criterion.

\begin{table}[H]
\centering\small\addtolength{\tabcolsep}{-3pt}
\centering
\begin{tabular}{c|c|c|c|c|c}  
\hline  
Tensor & Method & Res.norm& Error & Iter & CPU time \\ \hline 
$\mathcal{A}$ &t-PM & 2.14e-15   & 2.27e-15     & 537 & 0.382 \\ \hline 
$\mathcal{C}$ &t-PM & 3.70e-14  & 1.06e-14  & 2035   &  0.886    \\ \hline
${\tt complex(10,10,10)}$ & t-PM &  3.47e-14  & 2.56e-14  & 849 & 0.707 \\ \hline
\end{tabular}	
\caption{Residual norm, Error, number of iterations, and CPU time for t-PM.}\label{tab 2}
\end{table}

The inverse t-power method is implemented by Algorithm \ref{alg 4} with a stopping 
criterion that is analogous to the one for Algorithm \ref{alg 3}. Here we seek to 
determine the eigenslice associated with the eigentube closest to the tubal shift 
$\bm{\sigma}$.  The shifts $\bm{\sigma}$ used for for the tensors $\mathcal{A}$ and 
${\tt complex(10,10,10)}$ are the tube with $\bm{\sigma}(1,1,1)=10^{-5}$ and 
remaining entries zero, and the tube with $\bm{\sigma}(1,1,1)=10^{-3}$ and remaining 
entries zero, respectively.

\begin{table}[h!]
\centering
\centering\small\addtolength{\tabcolsep}{-3pt}
\begin{tabular}{c|c|c|c|c}
\hline
Tensor & Method & Res. norm &Error & Iter\\ \hline
$\mathcal{A}$  &t-SIPM & 3.81e-15& 2.91e-16  &32  \\ \hline 
${\tt complex(10,10,10)}$  & t-SIPM &  4.33e-16 &  6.65e-15 & 422 \\ \hline
\end{tabular}
\captionof{table}{Residual norm, Error, and number of iterations for t-SIPM.}\label{tab 3}
\end{table}

\subsection{Deflation}
We examine the performance of the previously discussed deflation techniques in order to 
approximate more than one eigenpair. Recall that the deflation methods differ in how we 
choose the lateral slice $\vec{\mathcal{V}}$;\; see \eqref{deflation}. We choose 
$\vec{\mathcal V}$ as the $i^{th}$ eigenslice (DE), as the $i^{th}$ t-Schur lateral slice
(DS), or as the $i^{th}$ left eigenslice (DLE), and we report the results using DE, DS, 
and DLE methods for the tensors $\mathcal{A}$ and ${\tt real}\left(10,10,10\right)$. To 
approximate the $i^{th}$ eigentube, with $i>1$, we apply the power method to the new 
tensor
$\mathcal{A}_i = \mathcal{A}_{i-1}-\bm{\lambda}_i * \vec{\mathcal{U}}_i *
\vec{\mathcal{V}}_i^H,
$
where $\mathcal{A}_1=\mathcal{A}$, $\bm{\lambda}_i$ is the previous approximated eigentube 
and $\vec{\mathcal{V}}_i$ is chosen as the previously approximated eigenslice (for the DE
method) or as the left eigenslice (for the DLE method). We choose
$
\mathcal{A}_i = \mathcal{A}_{i-1}-\bm{\lambda}_i * \vec{\mathcal{Q}}_i *
\vec{\mathcal{Q}}_i^H,
$
where $\vec{\mathcal{Q}}_i$ is the $i^{th}$ t-Schur lateral slice, when applying the DS 
method. We approximate the first $k$ eigenpairs of a given tensor 
$\mathcal{F}\in\mathbb{K}^{p\times p}_n$ and evaluate the residual norm
\[
\text{Res.norm}=\Vert \mathcal{F}*\mathcal{U}_k -\mathcal{U}_k*\mathcal{D}_k\Vert_F,
\]
where $\mathcal{U}_k\in\mathbb{K}^{p\times k}_n$ contains computed approximations of the 
eigenslices associated with the first eigentubes of $\mathcal{F}$, and 
$\mathcal{D}_k\in\mathbb{K}^{k\times k}_n$ is an f-diagonal tensor which on its diagonal 
has the computed approximations of the first $k$ eigentubes of $\mathcal{F}$.

Table \ref{tab 5} reports the Error and the residual norm when approximating several 
eigentubes of the tensors $\mathcal{A}$ and ${\tt real}(10,10,10)$ by the DE, DS, and DLE
methods. Our aim is to approximate Num>1 eigentubes of the tensor.

\begin{table}[h!]
\centering\small\addtolength{\tabcolsep}{-5pt}
\begin{tabular}{c|c|c|c|c|c|c|c|c|c|c}
\hline
\multirow{2}{*}{Tensor} & \multirow{2}{*}{Num} & \multicolumn{3}{l|}{DE} & 
\multicolumn{3}{l|}{DLE}  & \multicolumn{3}{l}{DS}\\ 
\cline{3-11} & & Error & Res. norm & time & Error & Res. norm & time & Error & Res. norm &
time \\ \hline
\multirow{2}{*}{$\mathcal{A}$}& 3 &  4.58e-15 & 7.13e-15 & 0.675 &4.19e-15 
&7.51e-15& 1.237 & 4.57e-15&   3.27e-15 & 0.668\\ 
& 5 &  4.79e-15  &6.87e-15& 0.975 &4.58e-15 &7.63e-15 &1.408 & 4.83e-15  & 3.43e-15& 0.840\\ \hline
\multirow{2}{*}{${\tt real}(10,10,10)$}& 4 &8.06e-13 & 2.44e-13 & 0.642 &8.09e-13& 2.63e-13 &1.246 & 8.19e-13 & 2.54e-13& 0.641 \\ 
 & 6& 8.10e-13 & 2.98e-13 &1.279 &8.30e-13 &2.91e-13 &1.955 & 8.21e-13&  3.08e-13& 1.018\\
\hline
\end{tabular}
\caption{The Error, the residual norm, and the CPU time of DE, DS, and DLE methods, when
applied to the tensors $\mathcal{A}$ and ${\tt real}\left(10,10,10\right)$ to 
approximate Num eigentubes.}\label{tab 5}
\end{table}

Table \ref{tab 5}\; shows deflation with DS shifts to be faster than deflation with DE or 
DLE shifts. Figure \ref{fig def}\; illustrates the convergence behavior for each eigenpair.
We can see that the number of iterations required is the same for each eigenpair for all
shifts DE, DEL, and DS.

\begin{figure}[h!]
	\centering
	\centering\small\addtolength{\tabcolsep}{-3pt}
	\begin{tabular}{ccc}
		\includegraphics[width=1.8in, height=1.7in]{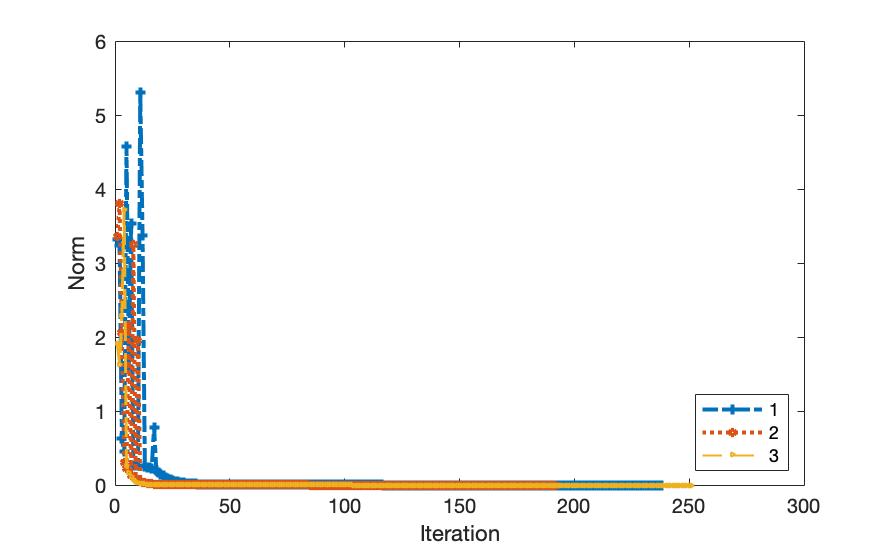}&
		\includegraphics[width=1.8in, height=1.7in]{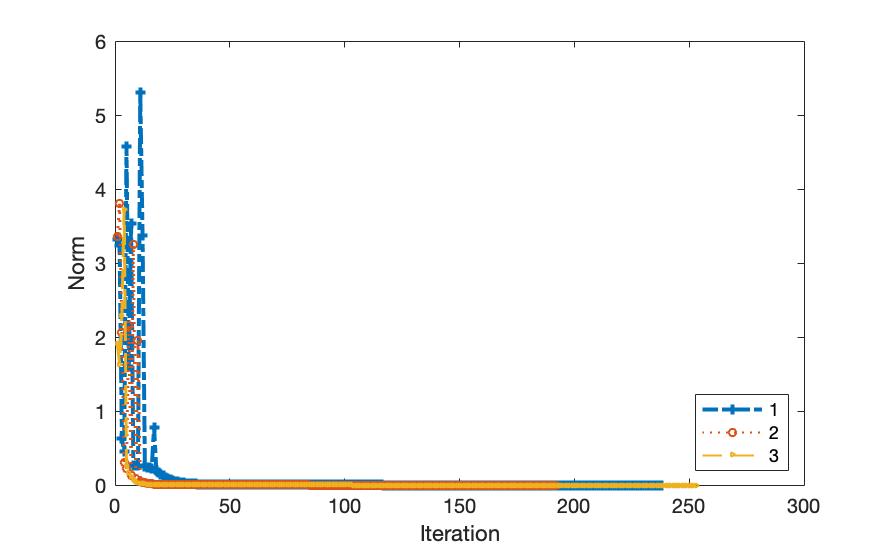}&
		\includegraphics[width=1.8in, height=1.7in]{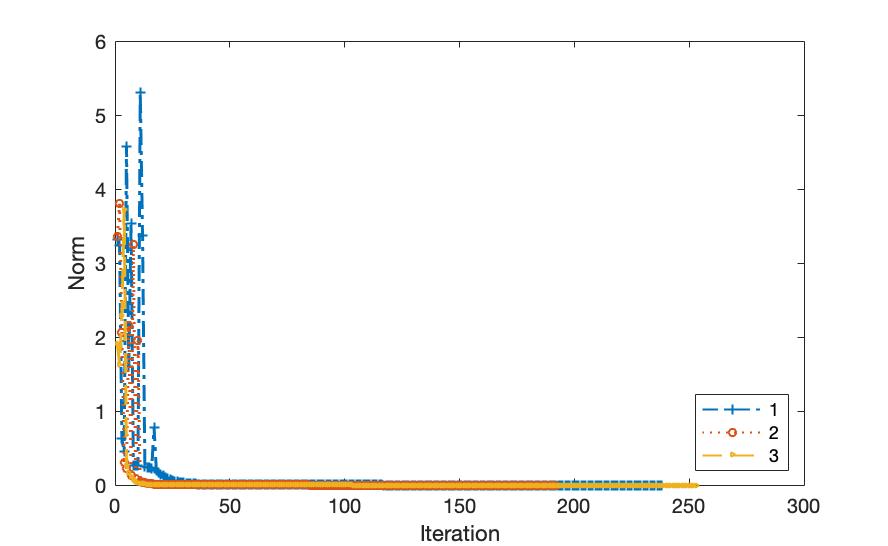}
	\end{tabular}
	\captionof{figure}{The behavior of the residual norm when we aimed approximate the first three eigenpairs of the tensor ${\tt real}\left(10,10,10\right)$, by the methods DE, DLE, and DS, from the left to the right.}\label{fig def}
\end{figure}

\subsection{Tensor subspace iteration}
This subsection illustrates the performance of the t-subspace iteration method when 
applied to approximate eigenslices associated with the $m$ eigentubes of largest norm
of a tensor $\mathcal{A}\in\K^{p\times p}_n$. The sequence of approximate eigenslices
$\{\mathcal{X}_k\}_{k\geq 1}$ generated by Algorithm \ref{alg 5}\ generically
converges to the $\tspan$ of the t-Schur lateral slices associated with the $m$ largest
eigentubes. Thus, the tensors
\[
\mathcal{R}_k=\mathcal{X}_k^H *\mathcal{A}*\mathcal{X}_k,\quad k=1,2,\ldots~,
\]
converge to an f-upper triangular tensor. This suggests the convergence criterion 
\begin{equation}\label{eq errplot}
\text{error}(k)={\tt t\text{-}tril}\left(\mathcal{R}_{k+1}-\mathcal{R}_k \right) \leq 
\text{tol},
\end{equation}
where the function ${\tt t\text{-}tril}$ returns the f-lower triangular portion of its
tensor argument. The tolerance used in all the experiments with the t-subspace method 
is $\text{tol}=10^{-15}$, and the maximum number of iterations is set to 
$\text{Itermax}=3000$. The initial tensor $\mathcal{X}_0$ has random entries chosen 
similarly as for the experiments with the t-power method. 

To approximate the first $m$ eigentubes, we compute the residual norm associated with 
the third-order tensor $\mathcal{F}\in\K^{p\times p}_n$,
\begin{equation}\label{resnorm2}
\text{Res.norm}=\Vert\mathcal{F}*\mathcal{U}_m -\mathcal{U}_m*\mathcal{R}_m\Vert_F,
\end{equation}
where $\mathcal{U}_m\in\K^{p\times m}_n$ is made up of computed approximations the t-Schur
lateral slices associated the $m$ largest eigentubes of $\mathcal{F}$, and 
$\mathcal{R}_m\in\K^{m\times m}_n$ is an upper  (or quasi-upper) f-triangular tensor, 
whose f-diagonal entries are approximations of the $m$ largest eigentubes of 
$\mathcal{F}$. 

Table \ref{tab s1} displays the computed results. We notice that t-MSI is the same as t-SI
when the power index $q$ is $1$. Figure \ref{fig s1} shows plots of the error 
\eqref{eq errplot} and of the residual norm at each iteration when seeking to approximate 
the four first eigentubes of the tensor ${\tt complex}(10,10,10)$ for the power indices 
$q\in\{1,4,6\}$. Table \ref{tab s1} indicates that increasing the power index $q$ has no 
effect on the accuracy, but reduces the total number of iterations.

\begin{table}[h!]
	\centering\small\addtolength{\tabcolsep}{-3pt}
	\centering
	\begin{tabular}{c|c|c|c|c|c}
		\hline Tensor & $q$ &  Error & Res.norm & Iter & CPU time\\
		\hline \multirow{2}{*}{$\mathcal{A}$} & 1 & 4.58e-15  & 2.62e-15  &490 & 0.288  \\
		& 4& 4.58e-15  & 2.36e-15  &129  &0.137  \\
		\hline \multirow{2}{*}{${\tt complex}(10,10,10)$} & 1& 7.40e-11& 2.10e-14&3000& 2.775\\
		& 4& 9.45e-14  &2.40e-14&956&1.090\\
		\hline 
	\end{tabular}
\captionof{table}{The values of the Error, the residual norm, the number of iterations, 
and the CPU time for different values of the power index $q$ obtained when using t-MSI for
computing the first four eigentubes.}\label{tab s1}
\end{table}

\begin{figure}[h!]
	\centering
	\centering\small\addtolength{\tabcolsep}{-3pt}
	\begin{tabular}{cc}
		\includegraphics[width=2.7in, height=2in]{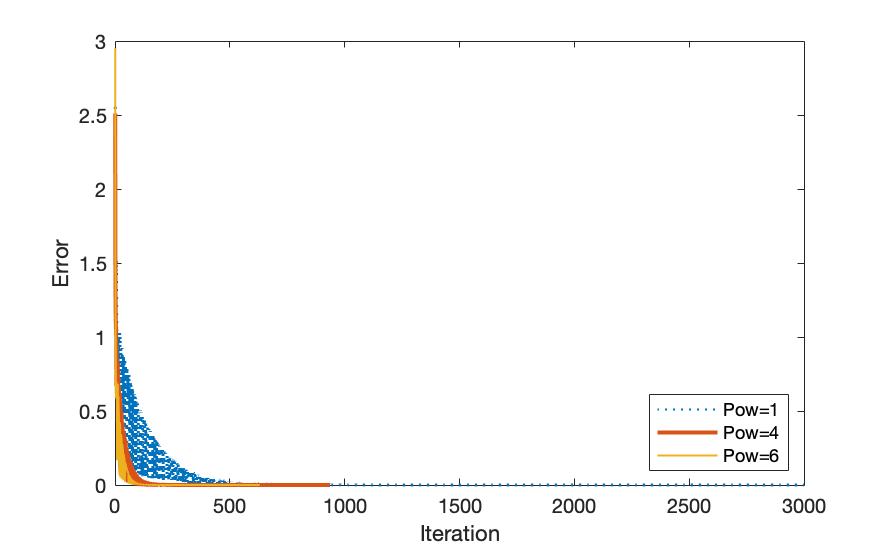}&
		\includegraphics[width=2.7in, height=2in]{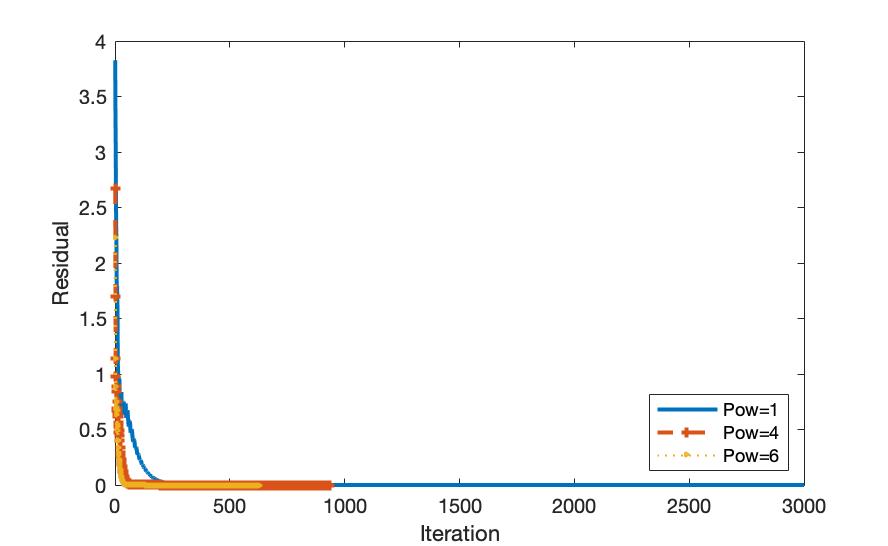}
	\end{tabular}
\captionof{figure}{The curves representing the evolution of the error and the residual 
norm for algorithm t-MSI with different values of the power index $q\in\{1,4,6\}$ applied
to the tensor ${\tt complex}\left(10,10,10\right)$, when we desire approximations of the 
first four eigentubes.}\label{fig s1}
\end{figure}

Figure \ref{fig s1} depicts the error norm and the residual norm for the t-MSI algorithm
for the values $1,\;4,\;6$ of the power index $q$. The results show fast convergence.

\subsection{The t-QR algorithm}
This subsection considers the t-QR algorithm t-QRHS. The algorithm is applied to the 
tensors $\mathcal{A}$ and $\mathcal{C}$. The tolerance is set to machine epsilon 
$2.2204\; 10^{-16}$; the maximum number of iterations allowed for Algorithm\ref{alg 7} is 
$\text{Itermax}=30000$. We compute the residual norm defined by expression 
\eqref{resnorm2}.

Table \ref{tab 10} reports the error norm \eqref{eq err}, the residual norm, the 
CPU time, and the number of iterations required to satisfy the stopping criterion of
the t-QRHS algorithm. The shifts have already been discussed after Algorithm \ref{alg 7}, 
but since the last two eigentubes of the tensor $\mathcal{C}$ are complex, we used in each
iteration the shift $\bm{\sigma}=\mathcal{H}_k(r,r,:)+i\mathcal{H}_k(r,r,:)$ for this
tensor.

\begin{table}[h!]
\centering\small\addtolength{\tabcolsep}{-3pt}
\centering
\begin{tabular}{c|c|c|c|c|r}
\hline 
Tensor & Method & Error & Res.norm & CPU time & Iter \\
	\hline {$\mathcal{A}$} & 
		t-QRHS & 1.3984e-14   & 1.5373e-14  &  0.169   & 61    \\
		\hline {$\mathcal{C}$} & 
		t-QRHS & 9.0322e-15 & 4.5962e-15  &  0.232  & 128  \\
		\hline
\end{tabular}
\captionof{table}{The Error, the residual norm, the CPU time, and the number of 
iterations needed for convergence of t-QRHS when applied to the tensors $\mathcal{A}$ and 
$\mathcal{C}$.}\label{tab 10}
\end{table}

\section{Conclusion and extensions}\label{sec7}
This work describes generalizations of eigenvalues and eigenvectors, referred to as
eigentubes and eigenslices, respectively, for third-order tensors using the t-product. 
Some properties of eigentubes and eigenslices are shown, and several methods for 
computing eigentubes and eigenslices are presented including the t-power method, the
t-inverse iteration method, and t-subspace iteration, as well as the t-QR algorithm. We 
also introduce deflation methods. Numerical tests illustrate the performance of these
methods.

\end{document}